\documentclass[a4paper,twoside]{article}
\usepackage{a4}
\usepackage{amssymb}
\usepackage{amsmath}
\usepackage[active]{srcltx}
\usepackage[dvipsnames]{xcolor} 
\usepackage[pagebackref,colorlinks,citecolor=blue,linkcolor=blue,urlcolor=blue]{hyperref}
\usepackage{amsthm}
\usepackage{upref}
\usepackage{graphicx}
\usepackage{subcaption}

\allowdisplaybreaks[2] 
\newcount\minutes \newcount\hours
\hours=\time
\divide\hours 60
\minutes=\hours
\multiply\minutes -60
\advance\minutes \time
\newcommand{\klockan}{\the\hours:{\ifnum\minutes<10 0\fi}\the\minutes}
\newcommand{\tid}{\today\ \klockan}
\newcommand{\prtid}{\smash{\raise 10mm \hbox{\LaTeX ed \tid}}}
\renewcommand{\prtid}{}
\makeatletter
\def\sectionmark#1{} 
\def\subsectionmark#1{}
\newcommand{\sectnr}{\ifnum \c@secnumdepth >\z@
   \thesection.\hskip 1em\relax \fi}
\def\@evenhead{\footnotesize\rm\thepage\hfil\leftmark\hfil\llap{\prtid}}
\def\@oddhead{\footnotesize\rm\rlap{\prtid}\hfil\rightmark\hfil\thepage}
\def\tableofcontents{\section*{Contents} 
 \@starttoc{toc}}
\makeatother
\makeatletter
\def\@biblabel#1{#1.}
\makeatother
\makeatletter
\let\Thebibliography=\thebibliography
\renewcommand{\thebibliography}[1]{\def\@mkboth##1##2{}\Thebibliography{#1}
\addcontentsline{toc}{section}{References}
\frenchspacing 
\setlength{\@topsep}{0pt}
\setlength{\itemsep}{0pt}%
\setlength{\parskip}{0pt plus 2pt}%
}
\makeatother
\makeatletter
\def\mdots@{\mathinner.\nonscript\!.%
 \ifx\next,.\else\ifx\next;.\else\ifx\next..\else
 \nonscript\!\mathinner.\fi\fi\fi}
\let\ldots\mdots@
\let\cdots\mdots@
\let\dotso\mdots@
\let\dotsb\mdots@
\let\dotsm\mdots@
\let\dotsc\mdots@
\def\vdots{\vbox{\baselineskip2.8\p@ \lineskiplimit\z@
    \kern6\p@\hbox{.}\hbox{.}\hbox{.}\kern3\p@}}
\def\ddots{\mathinner{\mkern1mu\raise8.6\p@\vbox{\kern7\p@\hbox{.}}%
    \raise5.8\p@\hbox{.}\raise3\p@\hbox{.}\mkern1mu}}
\makeatother
\makeatletter
\let\Enumerate=\enumerate
\renewcommand{\enumerate}{\Enumerate%
\setlength{\itemsep}{0pt}%
\setlength{\parskip}{0pt plus 1pt}%
\renewcommand{\theenumi}{\textup{(\alph{enumi})}}%
\renewcommand{\labelenumi}{\theenumi}%
}

\makeatother
\makeatletter
\def\@seccntformat#1{\csname the#1\endcsname.\quad}
\makeatother
\makeatletter
\long\def\@makecaption#1#2{%
  \vskip\abovecaptionskip
  \sbox\@tempboxa{ #1. #2}%
  \ifdim \wd\@tempboxa >\hsize
    #1. #2\par
  \else
    \global \@minipagefalse
    \hb@xt@\hsize{\hfil\box\@tempboxa\hfil}%
  \fi
  \vskip\belowcaptionskip}
\makeatother
\RequirePackage{ifthen}
\newcommand{\authortitle}[3]{\author{#1}\title{#2}%
   \ifthenelse{\equal{#3}{}}{\markboth{#1}{#2}}{\markboth{#1}{#3}}}
\newcommand{\auth}[2]{{#1, #2.}}
\newcommand{\art}[6]{{\sc #1, \rm #2, \it #3 \bf #4 \rm (#5), \mbox{#6}.}}
\newcommand{\artin}[3]{{\sc #1, \rm #2,  in #3.}}
\newcommand{\artprep}[3]{{\sc #1, \rm #2, #3.}}

\newcommand{\book}[3]{{\sc #1, \it #2, \rm #3.}}
\newcommand{\AND}{{\rm and }}

\newcommand{\arXiv}[1]{{\tt \href{https://arxiv.org/abs/#1}{arXiv:#1}}}
\RequirePackage{amsthm}
\newtheoremstyle{descriptive}%
  {\topsep}   
  {\topsep}   
  {\rmfamily} 
  {}          
  {\bfseries} 
  {.}         
  { }         
  {}          
\newtheoremstyle{propositional}%
  {\topsep}   
  {\topsep}   
  {\itshape}  
  {}          
  {\bfseries} 
  {.}         
  { }         
  {}          
\newtheoremstyle{remarkstyle}%
  {\topsep}   
  {\topsep}   
  {\rmfamily}  
  {}          
  {\itshape} 
  {.}         
  { }         
  {}          
\theoremstyle{propositional}
\newtheorem{thm}{Theorem}[section]
\newtheorem{prop}[thm]{Proposition}
\newtheorem{lem}[thm]{Lemma}
\newtheorem{cor}[thm]{Corollary}
\newtheorem*{ass}{General Assumptions}
\theoremstyle{descriptive}
\newtheorem{deff}[thm]{Definition}
\newtheorem{remark}[thm]{Remark}
\newtheorem{openprob}[thm]{Open problem}
\makeatletter
\renewenvironment{proof}[1][\proofname]{\par
  \pushQED{\qed}%
  \normalfont
  \trivlist
  \item[\hskip\labelsep
        \itshape
    #1\@addpunct{.}]\ignorespaces
}{%
  \popQED\endtrivlist\@endpefalse
}
\makeatother
\newcommand{\setm}{\setminus}
\renewcommand{\emptyset}{\varnothing}
{\catcode`p =12 \catcode`t =12 \gdef\eeaa#1pt{#1}}      
\def\accentadjtext#1{\setbox0\hbox{$#1$}\kern   
                \expandafter\eeaa\the\fontdimen1\textfont1 \ht0 }
\def\accentadjscript#1{\setbox0\hbox{$#1$}\kern 
                \expandafter\eeaa\the\fontdimen1\scriptfont1 \ht0 }
\def\accentadjscriptscript#1{\setbox0\hbox{$#1$}\kern   
                \expandafter\eeaa\the\fontdimen1\scriptscriptfont1 \ht0 }
\def\accentadjtextback#1{\setbox0\hbox{$#1$}\kern       
                -\expandafter\eeaa\the\fontdimen1\textfont1 \ht0 }
\def\accentadjscriptback#1{\setbox0\hbox{$#1$}\kern     
                -\expandafter\eeaa\the\fontdimen1\scriptfont1 \ht0 }
\def\accentadjscriptscriptback#1{\setbox0\hbox{$#1$}\kern 
                -\expandafter\eeaa\the\fontdimen1\scriptscriptfont1 \ht0 }
\def\itoverline#1{{\mathsurround0pt\mathchoice
        {\rlap{$\accentadjtext{\displaystyle #1}
                \accentadjtext{\vrule height1.593pt}
                \overline{\phantom{\displaystyle #1}
                \accentadjtextback{\displaystyle #1}}$}{#1}}
        {\rlap{$\accentadjtext{\textstyle #1}
                \accentadjtext{\vrule height1.593pt}
                \overline{\phantom{\textstyle #1}
                \accentadjtextback{\textstyle #1}}$}{#1}}
        {\rlap{$\accentadjscript{\scriptstyle #1}
                \accentadjscript{\vrule height1.593pt}
                \overline{\phantom{\scriptstyle #1}
                \accentadjscriptback{\scriptstyle #1}}$}{#1}}
        {\rlap{$\accentadjscriptscript{\scriptscriptstyle #1}
                \accentadjscriptscript{\vrule height1.593pt}
                \overline{\phantom{\scriptscriptstyle #1}
                \accentadjscriptscriptback{\scriptscriptstyle #1}}$}{#1}}}}
\def\cprime{{\mathsurround0pt$'$}}
\newcommand{\Cpt}{{C_{s,p}}}
\newcommand{\Csp}{{C_{s,p}}}
\newcommand{\Bsp}{{B_{s,p}}}
\newcommand{\cpt}{{\capp_{s,p}}}
\newcommand{\csp}{{\capp_{s,p}}}

\DeclareMathOperator{\diam}{diam}
\DeclareMathOperator{\dist}{dist}
\DeclareMathOperator{\capp}{cap}
\newcommand{\grad}{\nabla}
\DeclareMathOperator{\Lip}{Lip}
\newcommand{\Lipc}{{\Lip_c}}
\newcommand{\loc}{_{\rm loc}}
\DeclareMathOperator*{\essliminf}{ess\,lim\,inf}
\DeclareMathOperator{\Tail}{Tail}
\newcommand{\bdry}{\partial}
\newcommand{\bdy}{\bdry}
\newcommand{\simge}{\gtrsim}
\newcommand{\simle}{\lesssim}
\newcommand{\al}{\alpha}
\newcommand{\be}{\beta}
\newcommand{\ga}{\gamma}
\newcommand{\de}{\delta}
\newcommand{\eps}{\varepsilon}
\newcommand{\la}{\lambda}
\newcommand{\Om}{\Omega}
\newcommand{\uhat}{\hat{u}}
\newcommand{\ub}{\bar{u}}
\newcommand{\ut}{\tilde{u}}

\renewcommand{\phi}{\varphi}
\newcommand{\p}{{$p\mspace{1mu}$}}
\newcommand{\R}{\mathbf{R}}
\newcommand{\Rn}{{\R^n}}
\newcommand{\eR}{{\overline{\R}}}
\newcommand{\clB}{\itoverline{B}}
\newcommand{\clE}{\itoverline{E}}
\newcommand{\LL}{\mathcal{L}}
\newcommand{\EE}{\mathcal{E}}
\newcommand{\uP}{\itoverline{P}}     
\newcommand{\vt}{\tilde{v}}
\DeclareMathOperator{\supp}{supp}
\newcommand{\K}{\mathcal{K}}
\newcommand{\A}{\mathcal{A}}
\newcommand{\clG}{\itoverline{G}}
\newcommand{\Vsp}{V^{s,p}}
\newcommand{\VspOm}{V^{s,p}(\Om)}
\newcommand{\Vspo}{V^{s,p}_0}
\newcommand{\VspoOm}{V^{s,p}_0(\Om)}
\newcommand{\Wsp}{W^{s,p}}
\newcommand{\Wsploc}{W^{s,p}\loc}
\newcommand{\Wp}{W^{1,p}}
\newcommand{\La}{\Lambda}
\newcommand{\Omc}{{\Om^c}}
\newcommand{\Gc}{G^c}
\newcommand{\WspRn}{W^{s,p}(\Rn)}
\newcommand{\Schwarz}{\mathcal{S}}
\newcommand{\Et}{\widetilde{E}}

\newcommand{\tcsp}{{\widetilde{\capp}_{s,p}}}
\newcommand{\At}{\tilde{\mathcal{A}}}

\numberwithin{equation}{section}
\newcommand{\imp}{\ensuremath{
\mathchoice{\quad \Longrightarrow \quad}{\Rightarrow}
                {\Rightarrow}{\Rightarrow}}}

\begin{document}

\authortitle{Anders Bj\"orn, Jana Bj\"orn and Minhyun Kim}
{Capacities, Wiener criteria and fine continuity \\ for nonlocal nonlinear equations}
{Capacities, Wiener criteria and fine continuity for nonlocal nonlinear equations}

\author{
Anders Bj\"orn \\
\it\small Department of Mathematics, Link\"oping University, SE-581 83 Link\"oping, Sweden\\
\it \small anders.bjorn@liu.se, ORCID\/\textup{:} 0000-0002-9677-8321
\\
\\
Jana Bj\"orn \\
\it\small Department of Mathematics, Link\"oping University, SE-581 83 Link\"oping, Sweden\\
\it \small jana.bjorn@liu.se, ORCID\/\textup{:} 0000-0002-1238-6751
 \\
 \\
 Minhyun Kim \\
 \it\small Department of Mathematics \& Research Institute for Natural Sciences, \\
 \it\small Hanyang University, 04763 Seoul, Republic of Korea \\
 \it \small minhyun@hanyang.ac.kr, ORCID\/\textup{:} 0000-0003-3679-1775
}

\date{Preliminary version, \today}
\date{}

\maketitle

\noindent{\small
{\bf Abstract}.
In this paper we study nonlocal nonlinear equations of
$s$-fractional \p-Laplacian type in open subsets of 
$\mathbf{R}^n$.
We investigate how the boundary regularity of solutions
depends on the parameters $s$ and $p$,
including the local case $s=1$.
Specifically, we 
show exactly when
regularity for $(s_1, p_1)$
implies regularity for $(s_2, p_2)$.
The proof relies on the equivalence between
Wiener criteria formulated with condenser and Sobolev capacities.
To establish this equivalence, we derive 
precise comparison estimates between the two capacities.

The Wiener integral defines thinness and the fine topology.
We show that every superharmonic function
associated with 
a nonlocal nonlinear operator
is finely continuous. 
Moreover, we prove that polar sets in this fractional setting
coincide with sets of zero capacity.
}

\medskip

\noindent {\small \emph{Key words and phrases}:
Boundary regularity, capacity, fine topology, fractional 
\p-Laplacian, nonlocal nonlinear equation, polar set, thinness, Wiener criterion.
}

\medskip

\noindent {\small \emph{Mathematics Subject Classification} (2020):
Primary:
31C45. 
Secondary:  
31C15, 
35R11, 
35J66. 
}

\medskip

\noindent {\small \emph{Funding}: 
A.~B. resp.\ J.~B. were supported by the Swedish Research Council,
grants 2020-04011 and 2024-04095 resp.\ 2022-04048.
M.~K. was supported 
by the National Research Foundation of Korea (NRF) grant
funded by the Korean government (MSIT) (RS-2026-25481961).
}



\section{Introduction}

In this paper we study solutions (and supersolutions)
for the nonlocal nonlinear equation $\LL u =0$ 
in open subsets of $\mathbf{R}^n$.
The operator $\LL$ is of the form
\begin{equation*}
\LL u(x) = 2 \,\mathrm{p.v.} \int_{\R^n} |u(x)-u(y)|^{p-2} (u(x)-u(y)) k(x, y) \,dy,
\quad 1<p<\infty,
\end{equation*}
where 
$k: \R^n \times \R^n \to [0, \infty]$, $n \ge 1$, is a symmetric measurable kernel
that satisfies the ellipticity condition
\begin{equation} \label{eq-comp-(x,y)}
\frac{\Lambda^{-1}}{|x-y|^{n+sp}} \leq k(x, y) \leq \frac{\Lambda}{|x-y|^{n+sp}},
\end{equation}
with $0<s<1$ and $\La \ge1$.
Note that $\LL$ is the fractional \p-Laplacian $(-\Delta_p)^s$ when 
$k(x,y)=|x-y|^{-n-sp}$.

The following important Wiener criterion for regular boundary points
was recently obtained by 
Kim--Lee--Lee~\cite{KLL23}, \cite{KLL25} for general $\LL$ as here, and 
independently by Bj\"orn~\cite{JBWien} for $(-\Delta)^s$  
(with~$p=2$)
using the Caffarelli--Silvestre extension~\cite{CafSil}.
Roughly, boundary regularity means that
the solution of the Dirichlet problem with any continuous exterior data
always attains its Dirichlet data as a limit at a certain boundary point,
see Definition~\ref{deff-reg-pt} and the discussion after it.

\begin{thm} \label{thm-Wiener}
\textup{(Wiener criterion \cite[Theorem~1.1]{KLL23}, 
\cite[Remark~1.5]{KLL25})}
Let $\Om$ be a bounded open set.
Then a boundary point $x_0 \in \bdy \Om$ is regular\/ 
\textup(with respect to $\LL$ and $\Om$\textup) if and only if 
\begin{equation} \label{eq-Wiener}
 \int_0^1 \biggl(
  \frac{\csp(\clB(x_0,r)\setm \Om,B(x_0,2r))}{r^{n-sp}}
          \biggr)^{1/(p-1)} \, \frac{dr}{r} = \infty,
\end{equation}
where
the condenser capacity $\csp$ is defined in Definition~\ref{deff-cpt} below.
\end{thm}

It follows directly that
regularity is independent of the kernel $k$, 
and only depends on $s$ and $p$.
A natural question is how it
depends on $s$ and $p$. 
We answer this question in the following way.
For  a comparison with local equations, the \p-Laplacian (based on $W^{1,p}$)
is (formally) included with $s=1$,
see Section~\ref{sect-different-sp} for more details.
In particular, if $s_2 \ge s_1$ and $p_2 \ge p_1$, then
\eqref{eq-imp-reg-sp-intro} holds, i.e.\ 
regularity is more difficult for smaller $s$ and $p$, but the 
full picture is 
relatively complicated.

\begin{thm} \label{thm-reg-inclusion-intro}
Consider $0 < s_j \leq 1 < p_j$, $j=1,2$, with $(s_1, p_1) \ne (s_2, p_2)$.
The implication
\begin{equation}   \label{eq-imp-reg-sp-intro}
\text{$x_0$ is regular for $(s_1,p_1)$}
\imp
\text{$x_0$ is regular for $(s_2,p_2)$}
\end{equation}
holds, for all bounded open sets $\Om$ and boundary points
$x_0 \in \bdy \Om$,
if and only if
any of the following  mutually disjoint cases holds\/\textup{:}
\begin{enumerate}
\item
  $s_2 p_2 > n$,
\item
  $s_1 p_1 = s_2 p_2 = n$ and $p_1> p_2$,
\item
  $s_1 p_1 < s_2 p_2 = n$,
\item
  $s_1 p_1 < s_2 p_2 < n$ and 
$\displaystyle
  \frac{s_1(p_2-1)+n}{p_2} 
  \le  
  \frac{s_2(p_1-1)+n}{p_1}.
$ 
\end{enumerate}  
\end{thm}  

Using the Wiener criterion this essentially follows from results by
Adams--Hed\-berg~\cite[Theorem~B]{AH84} (see also Adams--Hedberg~\cite[6.5.8]{AH}).
However, \cite{AH84} and~\cite{AH}
used a Wiener integral of the type 
\eqref{eq-int-Csp=infty} below with the Sobolev capacity $\Csp$,
rather than the condenser capacity $\csp$ as in \eqref{eq-Wiener}.
It may be 
well known to the experts (or at least was in the 1990s)
that these two Wiener type integrals diverge/converge simultaneously
(provided that
$sp \le n$), but there does not seem to be any published proof.
In the local case, for capacities associated with $W^{1,p}$, 
the equivalence follows
from Theorem~2.49 in Mal\'y--Ziemer~\cite{MZ}.
The proof is quite standard and easy for $sp<n$, while for $sp=n$ it 
requires the following precise 
comparison for the Sobolev and condenser capacity.
See Section~\ref{sec-capacity} for definitions,
characterizations and properties of the capacities.
In particular, we show there that both capacities are so-called
Choquet capacities,
which is used in the proof of Theorem~\ref{thm-superh-finecont}.

\begin{lem}   \label{lem-Csp-min-cap-intro}
Let $E\subset   \itoverline{B(x_0,r)}$,  where $0 < r \le 1$.
Then
\[
\Csp(E) \simeq \min\{ \csp(E, B(x_0,2r)), \Csp(B(x_0,r)) \}.
\]
\end{lem}

This key lemma
implies the following comparison 
between Wiener criteria with 
different capacities.
Because of its connection to fine topology 
we formulate it for more general sets as follows.

\begin{thm}  \label{thm-equiv-Wiener-int}
Let $x_0 \in \Rn$.
For $0<r<1$, let $E_r \subset \itoverline{B(x_0,r)}$   
be arbitrary 
nested sets such that $E_r \subset E_t$ if $0 <r <t<1$.
If $sp\le n$, then 
\begin{equation}   \label{eq-int-csp=infty}
  \int_0^1 \biggl(\frac{\csp(E_r,B(x_0,2r))}{r^{n-sp}}
     \biggr)^{1/(p-1)}       \frac{dr}{r} =\infty
\end{equation} 
holds if and only if
\begin{equation}   \label{eq-int-Csp=infty}
  \int_0^1 \biggl(\frac{\Csp(E_r)}{r^{n-sp}}
     \biggr)^{1/(p-1)}       \frac{dr}{r} =\infty.
\end{equation} 
\end{thm}

If $sp > n$ and
$E_r=\{x_0\}$ for all sufficiently small $r>0$, then the integral in~\eqref{eq-int-csp=infty} 
diverges by Lemma~\ref{lem-sp<=n}, while the one in \eqref{eq-int-Csp=infty} converges.
Thus,  \eqref{eq-int-csp=infty} and \eqref{eq-int-Csp=infty}
are \emph{not} equivalent  when $sp>n$.

The Wiener integral defines thinness and the fine topology,
see Section~\ref{sect-fine-top}.
It is well known that fractional Sobolev functions have representatives
which are finely continuous q.e.\ (i.e.\ except for a set with zero capacity),
see the discussion after Proposition~\ref{prop-Wsp-better-repr}.
Fractional $\LL$-superharmonic  
functions are essentially 
regularized supersolutions
(see Theorem~\ref{thm:KKP17}) and therefore should automatically 
be ``good'' representatives. 
The following result shows that they are even better. 

\begin{thm} \label{thm-superh-finecont}
If $u$ is $\LL$-superharmonic in an open set $\Om$, 
then $u$ is finely continuous in $\Om$.
In particular, if $sp >n$ then $u$ is continuous in $\Om$.
\end{thm}

There are many $\LL$-superharmonic functions (e.g.\ Green functions) which 
take the value $\infty$ at some points, but not in a too large set.
This lies behind the definition of
so-called polar sets: a set $E$ is \emph{$\LL$-polar} 
if there
is an $\LL$-superharmonic function $u$ in some open set $\Om \supset E$ 
such that $E \subset \{x \in \Om : u(x)=\infty\}$.
The following result says that $\LL$-polar sets coincide with sets of zero 
capacity.
Proposition~\ref{prop-polar-iff}\ref{polar-a} will be used when proving
Theorem~\ref{thm-superh-finecont} for $sp>n$.

\begin{prop} \label{prop-polar-iff}
\begin{enumerate}
\item \label{polar-a}
If $u$ is $\LL$-superharmonic  
in an open set $\Om$, then
\[
\Csp(\{x \in \Om:u(x)=\infty\})=0.
\]
\item \label{polar-conv}
Conversely, if $\Om$ is a bounded open set,  $E \subset \Om$ and
$\Csp(E)=0$, then there is an
$\LL$-superharmonic function $v$ in $\Om$
such that $v=\infty$ in $E$ and $v \ge 0$ in $\Rn$.
\end{enumerate}
\end{prop}

Some historical remarks on negligible sets and fine topologies in the framework of local equations are in order. 
In classical linear potential theory, Evans~\cite{Eva36} proved the existence of a potential 
that becomes infinite precisely on a prescribed compact set of capacity zero. 
Soon after, Brelot~\cite{Bre41} introduced the concept of polar sets to systematically treat small sets. 
The equivalence between polar sets and 
sets of zero capacity was established by Cartan~\cite{Car45}. 
Choquet~\cite{Cho57} later generalized Evans's result, proving that any $G_\delta$ 
set of zero capacity  
coincides with the set of infinities of a potential.

The notion of fine topology was introduced by Cartan~\cite{Car45} 
as the coarsest topology making all superharmonic functions 
(associated with the Laplace equation)
continuous. 
The deep connection between this fine topology and the quasicontinuity 
of classical superharmonic functions was subsequently established 
in the seminal work of Fuglede~\cite{Fug71}. 
The fine topology was extended to nonlinear theories by Meyers~\cite{Mey75}, 
Adams--Meyers~\cite{AM72}, Adams--Hedberg~\cite{AH84} and Hedberg--Wolff~\cite{HW83}. 

For equations of \p-Laplace type,
Proposition~\ref{prop-polar-iff}\ref{polar-a}  was proved by 
Lindqvist--Martio~\cite{LM88} (for $p=n$) 
and Heinonen--Kilpel\"ainen~\cite[Theorem~1.5]{HK88} 
(for all $1 < p \le n$), while
Proposition~\ref{prop-polar-iff}\ref{polar-conv} was obtained
by Kilpel\"ainen~\cite[Theorem~1.7]{Kilp89}.
In~\cite[Theorem~1.3]{Kilp99}, Kil\-pe\-l\"ai\-nen extended Choquet's result~\cite{Cho57}
to equations of \p-Laplace type,
showing (in particular)
that for any $G_\de$-set $E \subset \Rn$ with zero \p-capacity,
there is a \p-superharmonic function $u$ such that 
$E=\{x \in \Rn:u(x)=\infty\}$,
cf.\ Open Problem~\ref{open-prob}.

The paper is organized as follows.
In Section~\ref{sec-sobolev}
we recall several relevant function spaces
and the notion of (weak) solutions of $\LL u=0$.
In Section~\ref{sec-capacity}, 
we recall the definitions of the condenser and Sobolev capacities,
prove that they are Choquet capacities
and establish several other properties.
We prove Lemma~\ref{lem-Csp-min-cap-intro} and
Theorem~\ref{thm-equiv-Wiener-int} in Section~\ref{sect-Wiener-criteria}.
Additional equivalent formulations of Wiener type criteria
are provided in Section~\ref{sect-Wiener-pizza}.

In Section~\ref{sect-fine-top}, we turn to the study of
fine topology, thin sets and quasicontinuity.
In Section~\ref{sect-polar}, the proofs of Theorem~\ref{thm-superh-finecont}
and Proposition~\ref{prop-polar-iff} are provided.
Finally, in Section~\ref{sect-different-sp},
we prove Theorem~\ref{thm-reg-inclusion-intro},
which shows how regularity varies with $s$ and $p$.
We end the paper with a discussion on
the conditions in Theorem~\ref{thm-reg-inclusion-intro}.

\section{Fractional Sobolev spaces
and solutions of \texorpdfstring{$\LL u=\nobreak0$}{Lu=0}}
\label{sec-sobolev}

\begin{ass}
Throughout the paper, we assume
that 
$0<s<1<p<\infty$,
that $k$  is a symmetric measurable kernel satisfying \eqref{eq-comp-(x,y)},
and that $\Om \subset \Rn$, $n \ge 1$,
is a nonempty  open set.
\end{ass}

To consider weak solutions of the equation
\begin{equation}\label{eq-Lu=0}
\LL u=0 \quad\text{in}~\Om,
\end{equation}
we need to define some function spaces.
For a measurable function $u: \Om \to \eR:=[-\infty,\infty]$ (which is
finite a.e.) we 
consider the fractional seminorm  
\begin{equation} \label{eq-seminorm}
  [u]_{\Wsp(\Om)}=
\biggl(  \int_{\Omega}\int_{\Omega} \frac{|u(x)-u(y)|^p}{|x-y|^{n+s p}} 
      \, dy\, dx\biggr)^{1/p}. 
\end{equation}
The fractional Sobolev space $\Wsp(\Omega)$ consists of the functions $u$ such that
the norm  
\begin{equation*}
\|u\|_{\Wsp(\Om)}^p:=\|u\|_{L^p(\Om)}^p+[u]_{\Wsp(\Om)}^p < \infty.
\end{equation*}
The above spaces and (semi)norms go under various names, 
such as fractional Sobolev, 
Gagliardo--Nirenberg, Sobolev--Slobodetski\u{\i} and Besov.

By $W^{s, p}_{\mathrm{loc}}(\Omega)$ we denote the space of functions that belong to $W^{s, p}(G)$ 
for every open $G \Subset \Omega$. 
As usual, by $E \Subset \Om$ we mean that $\itoverline{E}$
is a compact subset of $\Om$.
We refer the reader to Di Nezza--Palatucci--Valdinoci~\cite{DNPV12} for properties of these spaces.

As \eqref{eq-Lu=0} is a nonlocal equation, we also
need a function space that captures integrability of
functions in the whole of $\Rn$.
The \emph{tail space} $L^{p-1}_{sp}(\R^n)$ is given by
\begin{equation*}
L^{p-1}_{sp}(\R^n) = \biggl\lbrace u \text{ measurable}: 
  \int_{\R^n} \frac{|u(y)|^{p-1}}{(1+|y|)^{n+sp}} \,dy < \infty \biggr\rbrace.
\end{equation*}

It is often
convenient to deal with a larger class of test functions than $C_c^\infty(\Om)$.
For this purpose, we will use the spaces 
\begin{equation*}  
\Vsp(\Omega) := \biggl\lbrace u: \R^n \to \eR : 
u|_{\Omega} \in L^p(\Omega) \text{ and }  
\frac{|u(x)-u(y)|}{|x-y|^{n/p+s}} \in L^p(\Omega \times \R^n) \biggr\rbrace,
\end{equation*}
equipped with the norm
\begin{align*}
\|u\|_{V^{s, p}(\Omega)} 
&:= \bigl( \|u\|_{L^p(\Omega)}^p + [u]_{V^{s, p}(\Omega)}^p \bigr)^{1/p} \\
&:= \biggl( \int_{\Omega} |u(x)|^p \,dx + \int_{\Omega} \int_{\R^n} \frac{|u(x)-u(y)|^p}{|x-y|^{n+sp}} \,dy \,dx \biggr)^{1/p},
\end{align*}
and
\begin{equation*}
V^{s, p}_0(\Omega) := \overline{C_c^{\infty}(\Omega)}^{V^{s, p}(\Omega)},
\end{equation*}
with the convention that $u\equiv0$ outside $\Om$ for $u\in \Vsp_0(\Om)$.
The $\Wsp(\Rn)$-norm is an equivalent norm on $\Vsp_0(\Om)$,
and so 
are the seminorms $[\,\cdot\,]_{\Wsp(\Rn)}$ and $[\,\cdot\,]_{\Vsp(\Om)}$,
by the fractional Poincar\'e inequality (see e.g.\ \cite[(5.6)]{BBK1})
whenever $\Om$ is bounded.

Note that $\Vsp(\R^n)=\Wsp(\R^n)$
and that
\begin{equation} \label{eq-Lipc-inclusions}
\Lipc(\Om) \subset \Vspo(\Om) \subset
\Vsp(\Om) \subset
\Wsp(\Om) \cap L^{p-1}_{sp}(\R^n).
\end{equation}

See \cite[Remark~2.2]{KLL23}, \cite[Section~2]{BBK1}
and the references therein for further remarks 
on the spaces $\Vsp_0(\Om)$ and $\Vsp(\Om)$,
which were denoted by $\Wsp_0(\Om)$ and $\Vsp(\Om|\Rn)$
in~\cite{KLL23}.
 The space $V^{s,2}(\Om)$ was (with $p=2$) introduced by
 Servadei--Valdinoci~\cite{SV14} and independently by
 Felsinger--Kassmann--Voigt~\cite{FKV15}.

As usual,
$C_c^\infty(\Om)$ denotes
the space of $C^\infty$ functions with compact support in $\Om$.
Similarly, $\Lipc(\Om)$ denotes the space of Lipschitz 
functions with compact support in $\Om$.
The \emph{support} of a function $u$ is 
$\supp u =\overline{\{x:u(x)\ne0\}}$
and
$u_+=\max\{u,0\}$ and $u_-=\max\{-u,0\}$.

In this paper, there are many places where the functions need to be 
defined pointwise everywhere in a given set, and not just a.e. 
For convenience, we will therefore assume that all functions are defined
pointwise everywhere.

It is well known and easy to see that
the integrand in the seminorm~\eqref{eq-seminorm} decreases under truncations.
More generally, it follows from the following lemma that the spaces
$\Wsp(\Om)$ and  $\Vsp(\Om)$
are lattices,
i.e.\ if $u,v \in \Wsp(\Om)$ then $\max\{u,v\},\min\{u,v\} \in \Wsp(\Om)$,
and similarly for $\Vsp(\Om)$. 
This lemma will also be used to show the strong subadditivity
of the capacities (see Proposition~\ref{prop-strong-subadd}).
The lemma is elementary, 
see 
Kim--Lee--Lee~\cite[Lemma~A.3]{KLL25} for a proof of a more
general lemma.

\begin{lem} \label{lem-A3}
Let $a,a',b,b' \in \R$ be arbitrary.
Then
\[
   |{\max\{a,a'\}-\max\{b,b'\}}|^p +|{\min\{a,a'\}-\min\{b,b'\}}|^p
  \le    |a-b|^p+|a'-b'|^p.
\]
\end{lem}

To see that also $\VspoOm$ is a lattice, which is closed under certain truncations,
the following lemma is convenient.
For a proof see \cite[Lemma~2.5]{BBK1} 
or Costea~\cite[Lemma~2.2]{costea07}, 
which is similar but deals with a 
different nonlocal norm.

\begin{lem} \label{lem-truncation}
Assume that $g\in V^{s, p}(\Om)$ and $u_j \to u$ in ${V^{s, p}(\Om)}$ as $j\to\infty$.
Let 
\[
\ub=\min\{u,g\} \quad \text{and} \quad \ub_j=\min\{u_j,g\}.
\]
Then $\ub_j \to \ub$ in $V^{s, p}(\Om)$ as $j\to\infty$.
\end{lem}

\begin{lem} \label{lem-Vspo-lattice}
\begin{enumerate}
\item \label{Vspo-lattice}
$\VspoOm$ is a lattice.
\item \label{Vspo-lattice-min}
If $u \in \Vspo(\Om)$ and $l \ge 0$, then $\min\{u,l\} \in \Vspo(\Om)$.
\end{enumerate}
\end{lem}

\begin{proof}
\ref{Vspo-lattice-min}
We begin by showing that $\ut:=\min\{u,l\} \in \Vspo(\Om)$.
Let $u_j \in C^\infty_c(\Om)$ be
such that $u_j \to u$ in $\Vsp(\Om)$ as $j \to \infty$.
Thus
$\ut_j:=\min\{u_j,l\} \in \Lip_c(\Om) \subset \Vsp(\Om)$, 
cf.\ \eqref{eq-Lipc-inclusions}.
By Lemma~\ref{lem-truncation}, $\ut_j \to \ut$ in $\Vsp(\Om)$ as $j \to \infty$.
Since $\supp \ut_j \Subset \Om$, 
there are also mollifications $v_j\in C_c^\infty(\Om)$ of $\ut_j$
such that $v_j \to \ut$ in $\Vsp(\Om)$,
i.e.\ $\ut \in \Vspo(\Om)$.

\ref{Vspo-lattice} For $u,v \in \VspoOm$,
writing
\[
\min\{u,v\}=v+ \min\{u-v,0\}  \quad \text{and}  \quad
\max\{u,v\}=v- \min\{v-u,0\},
\]
together with \ref{Vspo-lattice-min},
concludes the proof.
\end{proof}

For measurable functions $u, v: \R^n \to \eR$ 
we define the quantity
\begin{equation*}
\mathcal{E}(u,v)=\int_{\R^n} \int_{\R^n} |u(x)-u(y)|^{p-2} (u(x)-u(y))(v(x)-v(y)) k(x, y) \,dy\,dx,
\end{equation*}
provided that   
the double integral exists and is finite.
Note that $\mathcal{E}(u, v)$ is well defined for 
$u \in \Wsp_{\mathrm{loc}}(\Om) \cap L^{p-1}_{sp}(\R^n)$ and $v \in C_c^\infty(\Om)$,
and that it is linear in the second argument.

The following definition of solutions and supersolutions of \eqref{eq-Lu=0} 
is standard by now.

\begin{deff}   \label{def-supersol}
Let $u \in W^{s, p}_{\mathrm{loc}}(\Om) \cap L^{p-1}_{sp}(\R^n)$.
Then $u$ is a 
(weak) \emph{solution} (resp.\ \emph{supersolution}) of $\LL u=0$ in $\Om$ if 
\begin{equation*}
\mathcal{E}(u, \varphi)  \ge 0 
\quad \text{for all  $\varphi \in C_c^{\infty}(\Omega)$
\quad
(resp.\ for all $0 \le \varphi \in C_c^{\infty}(\Omega)$).}
\end{equation*}
If $u \in C(\Om)$ is a solution in $\Om$, then $u$ is \emph{$\LL$-harmonic} in $\Om$.
\end{deff}

For simplicity we will just say that $u$ is a (super)solution in $\Om$, but we always
mean with respect to $\LL u=0$ in the weak sense of Definition~\ref{def-supersol}.
If $u$ is a solution, then there is an $\LL$-harmonic function
$v$ such that $v=u$ a.e.,
see e.g.\ Di Castro--Kuusi--Palatucci~\cite[Theorem~1.4]{DCKP16}.
The definition of solutions in~\cite[Theorem~1.4]{DCKP16} assumes $u\in\Wsp(\R^n)$, 
but the same proof shows that 
every solution in our sense has a representative which is
locally H\"older continuous in $\Om$.
The  supersolutions defined in  Korvenp\"a\"a--Kuusi--Palatucci~\cite{KKP17}
are only required to satisfy 
$u \in \Wsp_{\mathrm{loc}}(\Om) $ and 
$u_- \in L^{p-1}_{sp}(\R^n)$,
but  \cite[Lemma~1]{KKP17} shows that their
definition is equivalent to Definition~\ref{def-supersol}.

For functions $g \in \VspOm$ (and bounded $\Om$) 
the Dirichlet problem can be solved in the Sobolev sense. 
To make it precise we define Sobolev solutions $Hg$ as follows.

\begin{thm} \label{thm-ex-Hg}
\textup{(Kim--Lee~\cite[Theorem~4.9]{KL} 
and Korvenp\"a\"a--Kuusi--Palatucci~\cite[Theorem~9]{KKP17})}
Assume that $\Om$ is bounded.
Let 
$g \in \VspOm$.
Then there is a unique function $Hg := H_\Om g:\R^n \to \eR$ 
with the following properties\/\textup{:}
\begin{enumerate}
\item
  $Hg$ is $\LL$-harmonic in $\Om$,
\item
  $Hg-g \in \VspoOm$, and in particular
  $Hg \equiv g$ outside $\Om$.
\end{enumerate}
\end{thm}

\begin{deff} \label{deff-reg-pt}
Assume that $\Om$ is bounded.
A boundary point $x_0 \in \bdy \Om$ is \emph{regular} (with respect to $\LL$
and $\Om$)
if 
\begin{equation} \label{eq-deff-reg}
    \lim_{\Om \ni x \to x_0} H g(x)=g(x_0)
    \quad \text{for every } g \in \VspOm \cap C(\R^n).
\end{equation}  
 Otherwise we say that $x_0$ is \emph{irregular}.
\end{deff}  

This is the definition of regular boundary points
used in Kim--Lee--Lee~\cite{KLL23}.
Regularity can equivalently be formulated 
in terms of Perron solutions, see our paper~\cite{BBK1} 
and Lindgren--Lindqvist~\cite{LL17}.
Recall that regular boundary points are characterized by
the Wiener criterion in Theorem~\ref{thm-Wiener}.
Several other characterizations were given in \cite[Theorems~4.5 and~10.3]{BBK1}.
In particular, 
in~\eqref{eq-deff-reg} one can
equivalently consider all
bounded
$g \in \VspOm$ that are continuous at $x_0$, or even only $g(x) := \min\{|x-x_0|,1\}$.

In this paper, 
$A \simle A'$ (and $A' \simge A$) means that
$A\le CA'$
for some \emph{comparison constant}
$C >0$ independent of the quantities $A$ and $A'$.
If 
$A \simle A' \simle A$,
we write $A \simeq A'$.
For a ball 
\[
   B=B(x,r):=\{y \in \Rn : |y-x|<r\},
\]
we let $\la B=B(x,\la r)$ for $\la >0$.
By $E^c$ we denote the complement of the set $E$.

The following lemma is sometimes convenient.

\begin{lem} \label{lem-multiply-by-Lip}
Let $\eta$ be an $M$-Lipschitz function on $\Rn$ such that 
$0 \le \eta \le 1$.
If $u  \in \WspRn$, then
\begin{align} \label{eq-a-seminorm}
   [\eta u]_{\Wsp(\Rn)}^p 
    &\simle M^{sp} \|u\|_{L^p(\supp\eta)}^p + [u]_{\Wsp(\Rn)}^p, \\
\label{eq-a-norm}
     \|\eta u\|_{\Wsp(\Rn)} &\simle (M^{s}+1) \|u\|_{\Wsp(\Rn)}.
\end{align}
In particular, if $u_i \to u$ in $\Wsp(\Rn)$
then $\eta u_i \to \eta u \in \Wsp(\Rn)$.
\end{lem}

\begin{proof}
If $M=0$ there is nothing to prove, so assume that $M>0$.
Let $v=\eta u$.
Then
\begin{equation*}  
  |v(x)-v(y)| \le |u(x)|\,|\eta(x)-\eta(y)| + |u(x)-u(y)|  
\end{equation*}
and for all $x \in \Rn$,
\begin{align*} 
 \int_{\R^n}  \frac{|\eta(x)-\eta(y)|^p}{|x-y|^{n+sp}} \, dy
  &\simle  M^{p} \int_{B(x, 1/M)} |x-y|^{p(1-s)-n} \, dy \\
  &\quad + \int_{B(x, 1/M)^c} |x-y|^{-n-sp} \, dy 
\simle M^{sp}.
\end{align*}
Therefore, 
\begin{align*}
[v]_{\Wsp(\Rn)}^p
&\leq 2\int_{\supp \eta} \int_{\Rn} \frac{|v(x)-v(y)|^p}{|x-y|^{n+sp}} \,dy\,dx \\
&\simle M^{sp} \int_{\supp \eta} |u|^p\,dx + [u]_{\Wsp(\Rn)}^p,
\end{align*}
i.e.\ \eqref{eq-a-seminorm} holds.
Since obviously, $\|v\|_{L^p(\Rn)} \le \|u\|_{L^p(\Rn)}$,
we conclude that also \eqref{eq-a-norm} holds.
The last part now follows directly.
\end{proof}

\section{Capacities}\label{sec-capacity}

\emph{Recall the general assumptions from the beginning
of Section~\ref{sec-sobolev}.}

\medskip

Two types of capacities are used in this paper:
the Sobolev capacity and
the condenser capacity.
Since we need these capacities for noncompact sets we
define them as follows.

\begin{deff} \label{deff-cpt}
The \emph{condenser capacity} of a compact set $K \Subset \Om$ is given by
\[
    \cpt(K,\Om) = \inf_u {[u]_{W^{s,p}(\R^n)}^p},
\]
where the infimum is taken over all  $u\in C_c^\infty(\Om)$
such that 
$u \ge 1$  on $K$.
\end{deff}

\begin{deff}    \label{deff-Csp}
The \emph{Sobolev capacity} of a compact set $K \Subset \Rn$ is defined by
\begin{equation*} 
 \Cpt(K)   =\inf_u {\|u\|_{\Wsp(\Rn)}^p},
\end{equation*}
where the infimum is taken over all $u\in C_c^\infty(\Rn)$ such that 
$u \ge 1$ on $K$.
\end{deff}

Both capacities are then extended first to open and then to arbitrary sets 
in the usual way as follows:
For open sets $G$,
\begin{equation}   \label{eq-cap-G}
\begin{aligned}
\cpt(G,\Om) & = \sup_{\substack{K \text{ compact}\\ K \Subset G}} \cpt(K,\Om), && \text{if $G \subset \Om$}, \\
\Cpt(G) &= \sup_{\substack{K \text{ compact}\\ K \Subset G}} \Cpt(K), && \text{if $G \subset \Rn$}.
\end{aligned}
\end{equation}
Similarly, for arbitrary  sets~$E$,
\begin{equation}  \label{eq-cap-E}
\begin{aligned}
\cpt(E,\Om) &= \inf_{\substack{G \text{ open}\\ E \subset G \subset \Om}} \cpt(G,\Om), && \text{if $E \subset \Om$}, \\
\Cpt(E) &= \inf_{\substack{G \text{ open}\\ E \subset G}} \Cpt(G), && \text{if $E \subset \Rn$}.
\end{aligned}
\end{equation}

It is worth noticing that~\eqref{eq-cap-G} and~\eqref{eq-cap-E} do not change the 
definition of capacity for compact sets, 
cf.\  \cite[Section~5]{BBK1}.
In Kim--Lee--Lee~\cite[p.~29]{KLL25}, it is mentioned
that the fractional condenser capacity $\csp$
is a \emph{Choquet capacity} (i.e.\ 
it is monotone and satisfies the properties in Propositions~\ref{prop-Choq-Ki} 
and~\ref{prop-cspChoquet}),
with a reference
to Heinonen--Kilpel\"ainen--Martio~\cite{HeKiMa}. 
However, it seems far from obvious how the proof in~\cite{HeKiMa} 
(given in the nonlinear \emph{local} case associated with the $W^{1,p}$ Sobolev spaces) 
would carry over to the fractional setting.
So our next aim is to show that $\Csp$ and $\csp(\,\cdot\,,\Om)$ are 
indeed Choquet capacities 
(provided that $\Om$ is bounded).

We begin by showing the following result, which holds also for unbounded $\Om$.

\begin{prop} \label{prop-Choq-Ki}
If $K_1 \supset K_2 \supset \cdots$ are compact sets and $K=\bigcap_{i=1}^\infty K_i$,  then
\begin{alignat*}{2}
      \cpt(K,\Om)
          &= \lim_{i \to \infty} \cpt(K_i,\Om), &\quad& \text{if } K_1 \subset \Om, \\
      \Cpt(K) &          = \lim_{i \to \infty} \Cpt(K_i).
\end{alignat*}
\end{prop}

\begin{proof}
Let $G$ be an open set such that $K \subset G \subset \Om$.
Then $G\cup \bigcup_{i=1}^\infty K_i^c$ is an open cover  of
the compact set $K_1$.
Thus, there is a finite subcover, i.e.\ an $N$ such that
\[
K_1\subset G \cup \bigcup_{i=1}^N K_i^c
= G \cup K_N^c.
\]
As $K_N\subset K_1$, it follows that $K_N \subset G$.
So $\lim_{i \to \infty} \csp(K_i,\Om) \le \csp(G,\Om)$.
Taking the infimum over all open sets $G$ with $K \subset G \subset \Om$ and using
\eqref{eq-cap-E} shows that
$\lim_{i \to \infty} \csp(K_i,\Om) \le \csp(K,\Om)$.
The reverse inequality follows directly from monotonicity.

The second formula is shown similarly.
\end{proof}

To show that $\csp(\,\cdot\,,E)$ is a Choquet capacity we also need to show the
following continuity for increasing unions.

\begin{prop} \label{prop-cspChoquet}
Assume that $\Om$ is bounded.
Let 
$E_1  \subset E_2 \subset \dots  \subset E:=\bigcup_{i=1}^\infty E_i \subset \Om$. 
Then 
\[
      \csp(E,\Om)
          = \lim_{i \to \infty} \csp(E_i,\Om).
\]
\end{prop}

The following characterization of the condenser capacity
will be convenient in
the proof of Proposition~\ref{prop-cspChoquet} 
and may  be of independent interest.
The proof in~\cite[Proposition~6.2]{BBK1} uses 
nontrivial boundary regularity results.
We therefore seize the opportunity to give a 
more elementary proof.

\begin{prop} \label{prop-csp=tcsp}
\textup{(\cite[Proposition~6.2]{BBK1})}
Assume that $\Om$ is bounded.
If $E \subset \Om$, then
\begin{equation} \label{eq-cs=tsc-E}
   \csp(E,\Om) = 
   \tcsp(E,\Om):=\inf_{u \in \At(E,\Om)} {[u]_{\Wsp(\Rn)}^{p}},
\end{equation}
where $ \At(E,\Om) =\{u \in \VspoOm: u\ge 1 \text{ in  an open set containing } E\}$.

If moreover $E=G$ is open and $\csp(G,\Om)<\infty$, 
then the infimum is attained by a function $u \in \At(G,\Om)$
such that additionally $0 \le u \le 1$ everywhere.
\end{prop}

The letter $\A$ stands for ``admissible functions''.
By truncation (and Lemma~\ref{lem-Vspo-lattice}),
the infimum can equivalently be taken over
$u  \in\At(E,\Om)$  
such that (additionally) $0 \le u \le 1$ everywhere.

\begin{proof}[Proof of Proposition~\ref{prop-csp=tcsp}]
As both $\csp$ and $\tcsp$ are outer capacities (i.e.\ they 
satisfy \eqref{eq-cap-E})
it is enough to consider $E=G$ open.

First, we
show the $\le$ inequality in~\eqref{eq-cs=tsc-E}.
Let $K \Subset G$ be compact
and $\eta \in C_c^\infty(G)$ be a cutoff  function such
that $\eta=1$ on $K$ and $0 \le \eta \le 1$ everywhere. 
Let $u \in \At(G,\Om)$.
By truncation (and Lemma~\ref{lem-Vspo-lattice}), we can assume that 
$0 \le  u \le 1$.
Then there are $v_j\in C_c^\infty(\Om)$
such that $v_j \to u$ in $\Wsp(\Rn)$.
It follows from 
Lemma~\ref{lem-multiply-by-Lip} that also
$(1-\eta)(u-v_j) \to 0$ in $\Wsp(\Rn)$ and hence
\[
w_j:=v_j + \eta(u-v_j)\to u
\quad \text{in $\Wsp(\Rn)$}.
\]
Since $u(x)=1$ when $\eta(x) \ne 0$, we see that
$w_j=v_j+\eta(1-v_j)\in C_c^\infty(\Om)$. 
Moreover, $w_j=1$ on $K$. 
So, by Definition~\ref{deff-cpt},
\[
     \csp(K,\Om) \le [w_j]_{\Wsp(\Rn)}^p 
       \le ([u]_{\Wsp(\Rn)}+[w_j-u]_{\Wsp(\Rn)})^p
          \to [u]_{\Wsp(\Rn)}^p, 
\]
as $j \to \infty$.
Taking the infimum over all $u \in \At(G,\Om)$ and then 
the supremum over all compact $K\Subset G$,
it then follows from \eqref{eq-cap-G} that
\[
     \csp(G,\Om) \le \tcsp(G,\Om).
\]

Conversely,    
assume that $\csp(G,\Om)<\infty$ (otherwise there is nothing to prove).
Choose compact sets  $K_1\subset K_2 \subset \dots$
and functions $u_j \in C_c^\infty(\Om)$ so that $u_j\ge1$ on $K_j$,
\[
G=\bigcup_{j=1}^\infty K_{j} \quad \text{and}  \quad 
[u_j]_{\WspRn}^p < \csp(G,\Om) + 1/j.
\]

Since $\Om$ is bounded,  it follows that $\{u_j\}_{j=1}^\infty$ is a bounded
sequence in $\Wsp(\Rn)$
and $[\,\cdot\,]_{\Wsp(\Rn)}$ is an equivalent norm on $\VspoOm$
(see Section~\ref{sec-sobolev}).
By Banach--Alaoglu's theorem, there is a subsequence of $\{u_j\}_{j=1}^\infty$ 
which converges weakly to some function $v \in \Wsp(\Rn)$.
Since closed subspaces are weakly closed (as a consequence of 
the Hahn--Banach theorem), $v\in \Vspo(\Om)$.
Moreover, the subsequence converges weakly to $v$
also with respect to the equivalent norm
$[\,\cdot\,]_{\Wsp(\Rn)}$ on $\VspoOm$.
Hence
\[
   [v]^{p}_{\WspRn} \le  \liminf_{j \to \infty} {[u_j]^{p}_{\WspRn}} \le \csp(G,\Om).
\]

Using that $\phi \mapsto \int_E \phi\,dx$ 
is a bounded linear functional 
on $\Wsp(\R^n)$ for every bounded measurable set $E$,  
we obtain  that  $v\ge1$ a.e.\ in $G$.
After redefinition on a set of measure zero 
and truncation (Lemma~\ref{lem-Vspo-lattice}) we get
$0 \le v \le 1$ everywhere on $\Rn$ and
$v =1$ everywhere in $G$.
Thus 
$v \in \At(G,\Om)$ and so
\[
\tcsp(G,\Om) \le [v]_{\WspRn}^p \le \csp(G,\Om)
\le \tcsp(G,\Om),
\]
which also shows that
the infimum is attained by $v$.
\end{proof}

In the proof of Proposition~\ref{prop-cspChoquet}
 we will also use Mazur's lemma in the following form.
For a proof see e.g.\ 
Rudin~\cite[Theorem~3.12]{rudinFA}.

\begin{lem} \label{lem-Mazur}
\textup{(Mazur's lemma)}
If $x_j \to x$ as $j \to \infty$, weakly 
in a normed linear space $V$, and $\eps>0$, 
then there is a convex combination
$\sum_{j=1}^N a_jx_j$, with $a_j \ge 0$ and $\sum_{j=1}^N a_j=1$,
such that 
\[
      \biggl\| x- \sum_{j=1}^N
      a_jx_j \biggr\|_V < \eps.
\]
\end{lem}

\begin{proof}[Proof of Proposition~\ref{prop-cspChoquet}]
The $\ge$ inequality follows directly from monotonicity.
For the $\le$ inequality we can  
assume that $L:=\lim_{i \to \infty} \csp(E_i,\Om) < \infty$
(otherwise there is nothing to prove).
Let $\eps>0$.

By Proposition~\ref{prop-csp=tcsp},
for each $i=1,2,\dots$\,,
we can find  an open set $G_i \supset E_i$ and
a function $u_i \in \VspoOm$ such that 
$u_i= 1$ in $G_i$, $0 \le u_i \le 1$ everywhere, and
\[
   [u_i]_{\Wsp(\Rn)}^{p}         < \csp(E_i,\Om)+\eps 
\le L +\eps.
\]
Since $\Om$ is bounded, it follows that $\{u_i\}_{i=1}^\infty$ is a bounded
sequence in $\Wsp(\Rn)$.

By Banach--Alaoglu's theorem, 
there is a subsequence of 
$\{u_i\}_{i=1}^\infty$
which converges weakly to some function $v \in \Wsp(\Rn)$.
As in the proof of Proposition~\ref{prop-csp=tcsp},
$v\in \Vspo(\Om)$
and
\begin{equation}\label{eq-v-liminf}
   [v]^{p}_{\WspRn} \le  \liminf_{i \to \infty} {[u_i]^{p}_{\WspRn}} 
   \le     (L + \eps)^{1/p}.
\end{equation}

Applying Mazur's lemma repeatedly to subsequences of the
above subsequence, with indices starting at $j$,
we can find 
convex combinations 
$v_j=\sum_{i=j}^{N_j} a_{i,j} u_i \in \Vspo(\Om)$
(with $a_{i,j} \ge 0$ and $\sum_{i=j}^{N_j}a_{i, j}=1$)
such that 
\begin{equation}\label{eq-vj-v}
\|v_j-v\|_{\Wsp(\Rn)} < 2^{-j}\eps.
\end{equation}
Then
 $v_j = 1$ in the open set $G_j':=\bigcap_{i=j}^{N_j} G_i \supset E_j$.
Let
\begin{equation*}
w := v + \sum_{j=1}^\infty |v_j-v|.
\end{equation*}
By \eqref{eq-vj-v}, $w \in \Vspo(\Om)$.
Since 
$w \ge v+v_j-v = v_j =1$ 
in $G_j'$ for each $j$,
we get that $w \geq 1$ 
in the open set
$\bigcup_{j=1}^\infty G_j' \supset E$.
We thus see using 
\eqref{eq-v-liminf}, \eqref{eq-vj-v}
and Proposition~\ref{prop-csp=tcsp} that
\[
      \csp(E,\Om)^{1/p} 
    \le [w]_{\Wsp(\Rn)}
  \le  [v]_{\Wsp(\Rn)} + 2\eps 
\le (L+ \eps)^{1/p} +2\eps.
\]
Letting $\eps \to 0$ shows the $\le$ inequality.
\end{proof}

The proof of Proposition~\ref{prop-cspChoquet} above 
was inspired by the proofs of 
Kinnunen--Martio~\cite[Theorem~4.1]{KiMaNov} 
and Costea~\cite[Theorem~3.1]{costea07}.

As we have now shown that $\csp(\,\cdot\,,\Om)$ is a Choquet capacity
(if $\Om$ is bounded), 
it follows from Choquet's capacitability theorem 
(see e.g.\ Aikawa--Ess\'en~\cite[Part~2, Section~10]{AE}) that
the condenser capacity has the following inner regularity.

\begin{thm} \label{thm-cap-inner-reg}
Assume that $\Om$ is bounded.
If $E \subset \Om$ is a Borel\/ \textup(or Suslin\/\textup) set, then
\begin{equation}  \label{eq-cap-by-sup-K}
\cpt(E,\Om) = \sup_{\substack{K \text{ compact}\\ K \Subset E}} \cpt(K,\Om).
\end{equation}
\end{thm}

Hence $\cpt$ coincides for Borel sets with the capacity $\capp_{B^s_p}$
considered in Bj\"orn~\cite{JBWien}, which was only defined for Borel sets, using
Definition~\ref{deff-cpt} and \eqref{eq-cap-by-sup-K}.

The condenser capacity is also strongly subadditive in the following form.
For compact $E$ and $E'$ this was shown by
Kim--Lee--Lee~\cite[Theorem~5.2]{KLL25} in a more
general Orlicz setting.

\begin{prop} \label{prop-strong-subadd}
\textup{(Strong subadditivity)}
Assume that $\Om$ is bounded.
If $E,E' \subset \Om$, then
\[ 
  \csp(E \cup E',\Om) + \csp(E \cap E',\Om) \le \csp(E,\Om) + \csp(E',\Om).
\] 
\end{prop}

\begin{proof} 
Let $u \in \At(E,\Om)$ and $u' \in \At(E',\Om)$
in the notation of Proposition~\ref{prop-csp=tcsp}.
Then 
\[
u_{\max}:=\max\{u,u'\} \in \At(E \cup E',\Om)
\quad \text{and}  \quad
u_{\min}:=\min\{u,u'\} \in \At(E \cap E',\Om),
\]
by Lemma~\ref{lem-Vspo-lattice}.
Hence, by 
Proposition~\ref{prop-csp=tcsp}
and
Lemma~\ref{lem-A3},
\begin{align*}
  \csp(E \cup E',\Om) + \csp(E \cap E',\Om) 
  & \le   [u_{\max}]_{\Wsp(\Rn)}^p + [u_{\min}]_{\Wsp(\Rn)}^p\\
  &\le   [u]_{\Wsp(\Rn)}^p + [u']_{\Wsp(\Rn)}^p.
\end{align*}
Taking the infimum over all $u \in \At(E,\Om)$ and $u' \in \At(E',\Om)$,
and again appealing to 
Proposition~\ref{prop-csp=tcsp},
yields the desired inequality. 
\end{proof}

Together with Proposition~\ref{prop-cspChoquet}
this now directly leads to the countable subadditivity
of the condenser capacity.
In  \cite[Proposition~3.7]{BBK2} we gave 
another proof not relying on Proposition~\ref{prop-cspChoquet}.

\begin{prop} \label{prop-csp-subadd}
\textup{(Countable subadditivity, \cite[Proposition~3.7]{BBK2})}
Assume that $\Om$ is bounded.
If $E:=\bigcup_{i=1}^\infty E_i \subset \Om$, then 
\begin{equation*} 
      \csp(E,\Om)
          \le \sum_{i=1}^\infty \csp(E_i,\Om).
\end{equation*}
\end{prop}

\begin{proof}
By Proposition~\ref{prop-strong-subadd} and induction we see that
\[ 
      \csp\biggl(\bigcup_{i=1}^j E_i,\Om\biggr)
          \le \sum_{i=1}^j \csp(E_i,\Om).
\] 
The inequality then follows from Proposition~\ref{prop-cspChoquet}
upon letting $j \to \infty$.
\end{proof}

The following convenient formula 
for the Sobolev capacity corresponds to 
Proposition~\ref{prop-csp=tcsp} for the condenser capacity.

\begin{prop} \label{prop-Wsp-Csp-1}
\textup{(\cite[Lemma~5.7]{BBK1})}
Let $E \subset\Rn$.
Then
\begin{equation*} 
   \Cpt(E) = \inf_{u} {\|u\|^p_{\Wsp(\Rn)}},
\end{equation*}
where the infimum is taken over all $u \in \Wsp(\Rn)$ such that
$u \ge 1$
in  an open set containing $E$.
\end{prop}

With this formula in hand, it is easy to modify 
Propositions~\ref{prop-cspChoquet}, \ref{prop-strong-subadd}, \ref{prop-csp-subadd}
and Theorem~\ref{thm-cap-inner-reg} (and their proofs) for the Sobolev capacity.
(In this case  no boundedness assumption is needed on the sets.)

The following result 
relates the two capacities $\cpt$ and $\Cpt$
and shows that they have the same zero sets.

\begin{lem} \label{lem-cp-Cp}
\textup{(\cite[Proposition~5.4]{BBK1})}
Assume that $\Om$ is bounded.
Let $E  \Subset \Om$.
Then  
\begin{equation*} 
\frac{\Csp(E)}{1+(\diam\Om)^{sp}} \simle
  \csp(E,\Om) \simle \biggl(1+\frac{1}{\dist(E,\Omc)^{p}}  \biggr) \Csp(E).
\end{equation*}  
In particular, $\Cpt(E)=0$ if and only if  $\cpt(E,\Om)=0$.
\end{lem}

The following lemma is from Kim--Lee--Lee~\cite[Lemma~2.17]{KLL23}.
The first part follows directly from Bj\"orn~\cite[Lemma~2.4]{JBWien}, together with
Heinonen--Kilpel\"ainen--Martio~\cite[Theorems~2.18 and 2.19]{HeKiMa}.
The last part is then a consequence of \eqref{eq-cap-E}.

\begin{lem}\label{lem-cap}
Let $x_0 \in \Rn$.
If $0<r\leq R/2$, then
\begin{equation*} 
\csp(B(x_0, r),B(x_0,R))
\simeq
\begin{cases}
r^{n-sp}, &\text{if }sp<n, \\
R^{n-sp}, &\text{if }sp>n, \\
(\log (R/r))^{1-p}, &\text{if }sp=n.
\end{cases}
\end{equation*} 
In particular,
\begin{equation} \label{eq-cap-sp>n}
\csp(\{x_0\}, B(x_0,R))
\simeq R^{n-sp} \quad \text{when } sp>n.
\end{equation}
\end{lem}

The following lemma is well known. 
It is a direct consequence of \eqref{eq-cap-E}
and Lemmas~\ref{lem-cp-Cp} and~\ref{lem-cap}.

\begin{lem} \label{lem-sp<=n}
Assume that $\Om$ is bounded.
Let $x_0 \in \Om$.
Then the following are equivalent\/\textup:
\begin{enumerate}
\item
$sp \le n$,
\item
$\Csp(\{x_0\})=0$, 
\item
$\csp(\{x_0\},\Om)=0$.
\end{enumerate}
\end{lem}

\begin{remark} \label{rmk-Bessel-cap}
There are several equivalent definitions of fractional Sobolev spaces 
and of the related capacities.
More precisely, in 
Triebel~\cite[Theorem p.~172, Theorem p.~189 and Remark~4 pp.~189--190]{Triebel95}
it is shown that our space $\Wsp(\R^n)$ coincides 
(with comparable norms)
with the 
Besov space $B^{p,p}_s(\R^n)$
(as defined in~\cite[Definition~4.1.1]{AH} or~\cite[(1), p.~172]{Triebel95}), 
which is also the same as the Triebel--Lizorkin space $F^{p,p}_{s}(\R^n)$.

By
Proposition~\ref{prop-Wsp-Csp-1},
the infimum
in Definition~\ref{deff-Csp} can equivalently be taken over functions in the
Schwarz class $\Schwarz$ (of rapidly decreasing $C^\infty$ functions).
It thus follows from 
Adams--Hedberg~\cite[Definition~4.4.2 and Propositions~4.4.3--4.4.4]{AH}
that
the Bessel potential capacity (denoted
by $C_{s,p}$ in \cite{AH}) is comparable to our Sobolev capacity $\Csp$
for compact sets.
Since the Bessel potential capacity is extended from compact
to arbitrary sets as in~\eqref{eq-cap-G}--\eqref{eq-cap-E}
(see \cite[Definition~2.2.6]{AH}),
the comparability holds for arbitrary sets.

The Bessel potential norm~\cite[Definition~2.2.6]{AH}
is \emph{not} comparable to our fractional Sobolev norm, 
so it is quite remarkable  (as they say in \cite[p.~107]{AH}) 
that these capacities are comparable.
Note also that properties like those in 
Propositions~\ref{prop-Choq-Ki}, \ref{prop-cspChoquet}, 
\ref{prop-strong-subadd}, \ref{prop-csp-subadd}
and Theorem~\ref{thm-cap-inner-reg}
are not necessarily preserved when changing to a comparable
capacity.
\end{remark}

\section{Proofs of Lemma~\ref{lem-Csp-min-cap-intro} and
Theorem~\ref{thm-equiv-Wiener-int}}
\label{sect-Wiener-criteria}

\emph{Recall the general assumptions from the beginning
of Section~\ref{sec-sobolev}.}

\medskip

The following lemma provides the key estimates for $sp<n$.

\begin{lem}   \label{lem-comp-cap}
Assume that $sp<n$ and $E\subset \clB$, where $B:=B(x_0,r)$. Then
\[
  \frac{\Csp(E)}{1+r^{sp}} \simle \csp(E,2B)
  \simle \csp(E,\R^n) \le  \Csp(E).
\]
In particular, all these  capacities have the same zero sets and
are comparable for small sets.
\end{lem}

\begin{proof}
By definitions~\eqref{eq-cap-G} and~\eqref{eq-cap-E}
it is enough to show these inequalities for compact sets $E=K$.
The first inequality follows from Lemma~\ref{lem-cp-Cp},
while the third inequality is trivial, because the class of admissible functions
is the same (in Definitions~\ref{deff-cpt} and~\ref{deff-Csp})
and the infimized quantity is larger.

For the second inequality, 
let
$v\in C^\infty_c(\R^n)$
be admissible for $\csp(K,\R^n)$
and take a $2/r$-Lipschitz cut-off function $\eta\in C^\infty_c(2B)$
such that $0\le\eta\le1$ everywhere
and $\eta=1$ in $B$.
Because
$v\eta$ is admissible for $\csp(K,2B)$ and
vanishes outside $2B$, it follows from 
Lemma~\ref{lem-multiply-by-Lip} (with $M=2/r$)
that      
\begin{equation} 
\csp(K,2B) \le [v\eta]_{W^{s,p}(\R^n)}^p  
\simle r^{-sp} \int_{2B} |v(x)|^p \,dx + [v]_{W^{s,p}(\R^n)}^p.
\label{eq-est-cap-v}                   
\end{equation} 

Since $sp<n$, the fractional Hardy inequality (see 
Maz{\cprime}ya--Shaposhnikova~\cite[Theorem~2]{MS02})
gives
\begin{equation*}  
r^{-sp} \int_{2B} |v(x)|^p \,dx \simle \int_{2B} \frac{|v(x)|^p}{|x|^{sp}} \,dx 
\simle [v]_{W^{s,p}(\R^n)}^p.
\end{equation*}
Inserting this into~\eqref{eq-est-cap-v} and taking the infimum over all
admissible $v$ concludes the proof.
\end{proof}

The following result complements Lemma~\ref{lem-comp-cap}.

\begin{thm}   \label{thm-Csp-min-cap}
Assume that $sp=n$.  
Let $E\subset \clB$,  where $B=B(x_0,r)$ and $r\le1$.
Then
\[
\Csp(E) \simge \min\{ \csp(E,2B), \Csp(B) \}.
\]
\end{thm}

A key step in the proof below is the Harnack inequality for the 
Bessel nonlinear potential, 
due to Adams--Meyers~\cite[Theorem~6.1]{AM72}.
To be able to use it, 
we will need some results from
Adams--Hedberg~\cite{AH} on Bessel nonlinear potentials
and from Stein~\cite{Stein70}
relating the Bessel potential 
norm to our fractional Sobolev norm.

\begin{proof}
We can assume that $E=K$ is compact and that $x_0=0$.

First, assume that $n\ge2$ and thus 
$p >2$.
Remark~\ref{rmk-Bessel-cap}
implies that $\Csp$ is comparable to
the Bessel potential capacity 
in Adams--Hedberg~\cite[Definition~2.2.6]{AH}.
(See \cite[Section~1.2.6]{AH} for the definition of the 
Bessel potential norm.)
The Bessel potential capacity is denoted by $\Csp$ in~\cite{AH},
but to avoid confusion with our fractional Sobolev capacity $\Csp$,
we will here denote it by $\Bsp$.

By \cite[Theorem~2.2.7]{AH} and the remark after it, there is a 
capacitary measure $\mu$ supported in $K$ such that 
\[
\Bsp(K) = \mu(K) = \|V_K\|_{L^{s,p}(\Rn)}^{p},
\quad \text{where } 
V_K:= G_s* (G_s* \mu)^{1/(p-1)}. 
\]
Here $G_s$ and $L^{s,p}(\Rn)$ are the Bessel kernel and the Bessel potential space
as in~\cite[(1.2.11) and 
(1.2.29)]{AH}.
Moreover, by \cite[Remark p.~21]{AH}, the Bessel nonlinear potential $V_K$ 
satisfies
\begin{equation*} 
V_K=1\text{ in } K\setm Z \quad \text{for some set $Z$ with } \Csp(Z)=\Bsp(Z)=0.
\end{equation*}
Theorem~6.1 in  Adams--Meyers~\cite{AM72}
(or \cite[Lemma~9.8.1]{AH}) shows that for every 
$x\in A_r:= 4B\setm \overline{2B}$,
\[
\sup_{B(x,r/4)} V_K \simle V_K(x). 
\] 
Covering the annulus $A_r$ by at most $N$ such balls $B(x,\tfrac14r)$
(with $N$ only depending on the dimension $n$), together with the connectedness
of $A_r$ (since $n\ge2$), implies that
\begin{equation}   \label{eq-def-M-m}
M:= \sup_{A_r} V_K \simle \inf_{A_r} V_K =:m,
\end{equation}
with a comparison constant independent of $r$.
We shall now distinguish two cases.

1. If $M \ge\tfrac12$ then $m\simge 1$ by the Harnack property \eqref{eq-def-M-m}
and the function $V_K/m$  
is admissible for $\Bsp(B(x,r))$ as in \cite[Corollary~2.6.8]{AH},
for any ball with centre $x\in \bdy(3B)$.
Hence, 
\[ 
\Csp(B)  = \Csp(B(x,r)) \simeq \Bsp(B(x,r))
 \le \frac{\|V_K\|^p_{L^{s,p}(\R^n)}}{m^p}  
=\frac{\Bsp(K)}{m^p} \simle \Csp(K).
\] 

2. If $M<\tfrac12$, then let
\[
v = \begin{cases}
     2(V_K-M)_+ & \text{in } 4B, \\
        0 & \text{in } \R^n \setm 4B. 
   \end{cases}
\]
Note that $v=0$ outside $2B$.
The lower semicontinuity of $V_K$ 
(see \cite[Proposition~2.3.2]{AH}), together with the 
fact that $M < \tfrac12$,
implies that
$v \ge 1$
in an open neighbourhood 
of $K \setm Z$
and thus, by Proposition~\ref{prop-csp=tcsp},
\begin{equation*}
    \csp(K \setm Z,2B)  \le [v]^p_{\Wsp(\R^n)}.
\end{equation*}
It therefore follows from the subadditivity of the capacity
and Lemma~\ref{lem-cp-Cp} that
\begin{equation}  \label{eq-csp-le-wsp}
\csp(K, 2B) = \csp(K\setm Z, 2B) \le [v]^p_{\Wsp(\R^n)}. 
\end{equation}

Since $v$ vanishes outside $2B$, we have 
\[
[v]^p_{\Wsp(\R^n)} \le 2\int_{2B}\int_{\R^n} \frac{|v(x)-v(y)|^p}{|x-y|^{n+sp}} 
      \, dy\, dx.
\]
By splitting the inner 
integral into $3B$ and $\R^n\setm 3B$,
respectively, and
estimating the latter integral 
(with $v(y)=0$ for $y \in \Rn \setm 3B$) as 
\[
\int_{\R^n\setm 3B} \frac{|v(x)-v(y)|^p}{|x-y|^{n+sp}} \, dy
\simeq
r^{-sp}|v(x)|^p
\simeq
\int_{4B\setm 3B} \frac{|v(x)-v(y)|^p}{|x-y|^{n+sp}} \, dy,
\]
we obtain that
\[
[v]^p_{\Wsp(\R^n)} 
\simle [v]^p_{\Wsp(3B)} + \int_{2B}\int_{4B\setm 3B} \frac{|v(x)-v(y)|^p}{|x-y|^{n+sp}} 
      \, dy\, dx
\simle [v]^p_{\Wsp(4B)}.
\]

Noting that $v=2(V_K-M)_+$ in $4B$
and using \eqref{eq-csp-le-wsp}
then gives
\[
\csp(K, 2B) \simle 
[(V_K-M)_+]^p_{\Wsp(4B)}
\le\|V_K\|_{\Wsp(\R^n)}^{p}.
\]
Since $p\ge2$, Theorem~V.5(A) and its proof on pp.~155--157 in Stein~\cite{Stein70}
imply that the $\Wsp(\R^n)$-norm
is dominated by the 
Bessel potential $L^{s,p}(\R^n)$-norm \cite[Section~1.2.6]{AH}.
Thus
\begin{equation*} 
\csp(K, 2B) \simle \|V_K\|^p_{L^{s,p}(\R^n)} 
=\Bsp(K) \simeq \Csp(K),
\end{equation*}
which proves the statement for $n\ge2$.

Finally, consider the case $n=1$.
Let $u\in C_c^\infty(\R)$ be such that $u\ge1$ on~$K$.
If 
$m:= u_+(2 r) + u_+(-2 r) \ge \tfrac12$, then 
the continuous even function
\[
\ub(x) = \begin{cases}
   \displaystyle  \frac{u_+(x)+u_+(-x)}{m} & \text{in } \R \setm 2 B,  \\
           1 & \text{in } 2 B, 
   \end{cases}
\]
is admissible for $\Csp(B)$, 
as in Proposition~\ref{prop-Wsp-Csp-1}, and hence
\[  
\Csp(B) 
\le \|\ub\|^p_{\Wsp(\R)} 
\simle \|u\|^p_{\Wsp(\R)}.
\]   

On the other hand, if $m< \tfrac12$ then the continuous even function
\[
v(x) = \begin{cases}
  2 \bigl(u_+(x)+u_+(-x)-m \bigr)_+ & \text{in } 2 B, \\
        0 & \text{in } \R \setm 2 B,
   \end{cases}
\]
is admissible for  $\csp(K,2B)$, as in Proposition~\ref{prop-csp=tcsp},
and hence
\[
  \csp(K,2B) \le  
[v]^p_{\Wsp(\R)}  \le
\|u\|^p_{\Wsp(\R)}.
\]
Combining both cases and taking the infimum over 
all $u\in C_c^\infty(\R)$ admissible for $\Csp(K)$
concludes the proof also for $n=1$.
\end{proof}

\begin{remark}
The proof of Theorem~\ref{thm-Csp-min-cap} applies (with some 
simplifications) also to the capacities associated with the 
Sobolev space $W^{1,p}$, 
as in Mal\'y--Ziemer~\cite[Chapter~2]{MZ}.
Since $W^{1,p}$ coincides with the Bessel potential space $L^{1,p}$,
by Calder\'on~\cite[Theorem~7]{Calderon61} 
(or Adams--Hedberg~\cite[Theorem~1.2.3]{AH}),
our proof of Theorem~\ref{thm-Csp-min-cap} gives an alternative proof
of \cite[Lemma~2.11]{MZ}. 
\end{remark}

\begin{proof}[Proof of Lemma~\ref{lem-Csp-min-cap-intro}]
Let $B=B(x_0,r)$.
By Lemma~\ref{lem-cp-Cp}, $\Csp(E) \simle \csp(E,2B)$ (because $r \le 1$).
Using also that $E \subset B$ 
yields 
\[
\Csp(E) \simle \min\{ \csp(E,2B), \Csp(B) \}.
\]
The reverse inequality follows from 
Lemma~\ref{lem-comp-cap} when $sp<n$, from 
Theorem~\ref{thm-Csp-min-cap} when $sp=n$,
while 
\[
\Csp(B) \simle 1 \simle \Csp(E) 
\quad \text{if } sp>n \text{ and } E \ne \emptyset.
\]
(When $E=\emptyset$ the estimate in Lemma~\ref{lem-Csp-min-cap-intro}
is trivial.)
\end{proof}

\begin{proof}[Proof of Theorem~\ref{thm-equiv-Wiener-int}]
When $sp<n$, this follows directly from Lemma~\ref{lem-comp-cap} upon noting
that  $1\le 1+r^{sp} \le 2$ for $0<r<1$.

Assume therefore that $sp=n$.
The implication   
\eqref{eq-int-Csp=infty}$\imp$\eqref{eq-int-csp=infty}
follows
from the first inequality in Lemma~\ref{lem-cp-Cp}.
For the converse implication,
we will use Theorem~\ref{thm-Csp-min-cap}.
Write $B_r:=B(x_0,r)$.
If 
\begin{equation}   \label{eq-min=csp}
\csp(E_r,B_{2r}) \le 
\Csp(B_{r}) 
\end{equation}
for all sufficiently small $r>0$,  then also 
\[
\csp(E_r,B_{2r}) \simle \Csp(E_r), 
\]
by Theorem~\ref{thm-Csp-min-cap},
and \eqref{eq-int-csp=infty} clearly implies \eqref{eq-int-Csp=infty}.

If \eqref{eq-min=csp} fails for 
infinitely many radii tending to zero, 
then choose among them 
$r_j \le \tfrac12$, $j=1,2,\dots$\,, so that 
\(
\log(1/r_{j+1}) \ge 2\log(1/r_{j}),
\)
i.e.\ $r_{j+1}\le r_j^2$, and
\[
\csp(E_{r_j},B_{2r_j}) 
> \Csp(B_{r_j}).
\] 
It then follows from Theorem~\ref{thm-Csp-min-cap}
and Adams--Hedberg~\cite[Proposition~5.1.3]{AH} that
\[
\Csp(E_{r_j}) 
\simge \Csp(B_{r_j}) 
\simeq (\log(1/r_j))^{1-p}.
\] 
Therefore, as the sets $E_r$ are nested and $n=sp$, we have
\begin{align*}
  \int_0^1 \biggl(\frac{\Csp(E_r)}{r^{n-sp}}\biggr)^{1/(p-1)} \frac{dr}{r} 
   &\simge \sum_{j=1}^\infty 
\frac{1}{\log(1/r_{j+1})}
\int_{r_{j+1}}^{r_j} \frac{dr}{r} \\
   &= \sum_{j=1}^\infty \frac 
{\log(1/r_{j+1})-\log(1/r_{j})}
{\log(1/r_{j+1})} 
\ge \sum_{j=1}^\infty \frac12
= \infty.\qedhere
\end{align*}
\end{proof}

\section{Equivalent Wiener type criteria}\label{sect-Wiener-pizza}

\emph{Recall the general assumptions from the beginning
of Section~\ref{sec-sobolev}.}

\medskip

The following result provides another equivalent formulation of the Wiener criterion.

\begin{prop}  \label{prop-half-pizza-sp<n}
Let $E\subset\R^n$, $x_0 \in \Rn$ and $B_r=B(x_0,r)$.
If $sp< n$, then the Wiener condition \eqref{eq-int-csp=infty}
{\rm(}with $E_r=E \cap \clB_r$\/{\rm)}
holds if and only if
\begin{equation}   \label{eq-half-pizza}
  \int_0^1 \biggl(\frac{\csp(E\cap (\clB_r \setm \clB_{r/2}),2B_r)}{r^{n-sp}}
     \biggr)^{1/(p-1)}       \frac{dr}{r} =\infty.
\end{equation} 
\end{prop}

\begin{proof}
The 
implication \eqref{eq-half-pizza}\imp\eqref{eq-int-csp=infty}
follows from the monotonicity of the capacity.
For the converse 
implication,
assume first that $p\ge2$. 
Since $1/(p-1)\le1$, we have by the subadditivity of the capacity, 
together with the elementary
inequality $\bigl( \sum_{i=0}^\infty a_i \bigr)^\ga \le \sum_{i=0}^\infty a_i^\ga$
for $0<\ga\le1$ and $a_i\ge0$, that 
\begin{align*}   
I   &:=  \int_0^1 \biggl(\frac{\csp(E\cap \clB_r,2B_r)}{r^{n-sp}}
     \biggr)^{1/(p-1)}       \frac{dr}{r}  \nonumber  \\
 &\le  \int_0^1 \biggl( \sum_{i=0}^\infty 
      \frac{\csp(E\cap (\clB_{2^{-i}r} \setm \clB_{2^{-i-1}r}),2B_r)}{r^{n-sp}}
              \biggr)^{1/(p-1)}  \frac{dr}{r}     \\
 &\le \sum_{i=0}^\infty  \int_0^1 \biggl( 
      \frac{\csp(E\cap (\clB_{2^{-i}r} \setm \clB_{2^{-i-1}r}),2B_r)}{r^{n-sp}}
              \biggr)^{1/(p-1)}  \frac{dr}{r} =: \sum_{i=0}^\infty I_i.
\end{align*} 
The change of variable $\rho=2^{-i} r$ in the last integral, together with the fact that
$\csp(A,2B_r) \le \csp(A,2B_{2^{-i}r})$ for every $A\subset \clB_{2^{-i}r}$, shows that
\[  
I_i \le 
\int_0^{2^{-i}} \biggl(  \frac{\csp(E\cap (\clB_{\rho} \setm \clB_{\rho/2}),2B_\rho)}
     {(2^{i}\rho)^{n-sp}}              \biggr)^{1/(p-1)}  \frac{d\rho}{\rho} 
\le \frac{I_0}{2^\frac{i(n-sp)}{p-1}}.
\] 
Since $n-sp>0$, summing over 
$i=0,1,\dots$\,, and inserting this into the above estimate of $I$ gives
\[
I_0 \le I \simle I_0
\]
and proves the statement for $p\ge2$.

For $1<p<2$ (and thus $1/(p-1)>1$), we instead use the fact that $I^{p-1}$
can be seen as an $L^{1/(p-1)}$-norm with respect to $dr/r$
on $(0,1)$.
The Minkowski inequality and the same 
change of variables
as for $p\ge2$ then give first $I^{p-1} \le \sum_{i=0}^\infty I_i^{p-1}$
and then $I_0^{p-1} \le I^{p-1} \simle I_0^{p-1}$.
\end{proof}

\begin{remark} \label{rmk-Wiener-sum}
The following lemma 
implies that the Wiener criterion~\eqref{eq-int-csp=infty}
can equivalently be formulated
using Wiener sums and that several closely
related {formulations} of the Wiener integral are equivalent.
In particular, 
the integration can equivalently be
taken over $(0,\de)$ for any $\de>0$,
and
the closed balls in the
Wiener condition~\eqref{eq-half-pizza} 
may equivalently be replaced by the corresponding open balls.
\end{remark}

\begin{lem}   \label{lem-cap-with-tr}
Let $x_0 \in \Rn$ and $B_\rho=B(x_0,\rho)$.
Also let $E \subset \clB_r$ and $1<t_1<t_2$.
Then
\[
\csp(E,B_{t_2r}) \le \csp(E,B_{t_1r}) \simle \csp(E,B_{t_2r}),
\]
with comparison constant only depending on $t_1$, $t_2$, $s$, $p$ and $n$.
\end{lem}

\begin{proof}
The first inequality is trivial.
For the second inequality, we may assume that $x_0=0$.
Let $\Et=\{x/r: x \in E\}$.
Then, by the change of variables $x \mapsto x/r$ and 
Lemma~\ref{lem-cp-Cp},
\begin{align*}
   \csp(E,B_{t_1r}) 
   &= r^{n-sp} \csp(\Et,B_{t_1})  
   \simle  r^{n-sp}\Csp(\Et) \\
   &\simle  r^{n-sp} \csp(\Et,B_{t_2}) 
   = \csp(E,B_{t_2r}),
\end{align*}
with comparison constants only depending on $t_1$, $t_2$, $s$, $p$ and $n$.
\end{proof}

For $sp=n$,
Proposition~\ref{prop-half-pizza-sp<n} needs to be modified in
the following way.

\begin{prop}  \label{prop-half-pizza-sp=n}
Let $E\subset\R^n$, $x_0 \in \Rn$ and $B_r=B(x_0,r)$.
If $sp= n$, then the Wiener condition \eqref{eq-int-csp=infty}
{\rm(}with $E_r=E \cap \clB_r$\/{\rm)}
holds if and only if
\begin{equation*} 
  \int_0^1 \csp\bigl(E\cap (\clB_r \setm \clB_{r^2}),2B_r\bigr)^{1/(p-1)}   \frac{dr}{r} =\infty.
\end{equation*} 
\end{prop}

\begin{proof}
The proof is the same as for Proposition~\ref{prop-half-pizza-sp<n}, 
with $I_i$ replaced by
\[
I'_i := \int_0^1 
\csp\bigl(E\cap ( B_{r^{2^i}} \setm B_{r^{2^{i+1}}}),2B_r \bigr)^{1/(p-1)}  \frac{dr}{r} \\
\le 2^{-i} I'_0.
\qedhere
\]
\end{proof}

\begin{remark} \label{rmk-half-pizza-sp=n}
Without the change in the statement of Proposition~\ref{prop-half-pizza-sp=n},
the equivalence of the Wiener conditions~\eqref{eq-int-csp=infty}
and \eqref{eq-half-pizza} fails when $sp=n$.
This is seen by (essentially) the same counterexample as
in Bj\"orn--Bj\"orn--Manolis~\cite[Example~7.2]{BBManolis}
(which is for the local case with $p=n$).
More precisely, for
\[
E:= \{0\} \cup \bigcup_{j=1}^\infty   B(x_j,\al_{j+1}),
\quad \text{with } 
\al_j = e^{-2^j} \text{ and }  x_j=(\al_j,0,\dots,0),
\]
we have
\[
E \cap (B_r \setm B_{r/2}) \subset \begin{cases}
                  B(x_j,\al_{j+1}) 
& \text{for }  \tfrac34 \al_j \le r \le \tfrac52 \al_j, \\
                  \emptyset   
& \text{otherwise.}   \end{cases}
\]
Since  the estimate for $\csp(B_r,B_R)$ with $sp=n$ in Lemma~\ref{lem-cap} 
is exactly the same as for the local capacity $\capp_p(B_r,B_R)$ with $p=n$,
namely $(\log(r/R))^{1-p}$,
the estimates and calculations in  \cite[Example~7.2]{BBManolis} apply verbatim
and show that the integral in \eqref{eq-half-pizza} converges,
while the one in~\eqref{eq-int-csp=infty} diverges. 
\end{remark}

\section{Fine topology, thin sets and quasicontinuity}
\label{sect-fine-top}

\emph{Recall the general assumptions from the beginning
of Section~\ref{sec-sobolev}.}

\medskip

The fine topology is defined as follows.
It depends on $s$ and $p$, but we make the dependence implicit in the notation.

\begin{deff}\label{deff-thinness}
A set $E$ is  \emph{thin} at $x$ if
\begin{equation*} 
\int_0^1\biggl(\frac{\csp(E\cap B(x,r),B(x,2r))}{r^{n-sp}}\biggr)^{1/(p-1)}
     \frac{dr}{r}<\infty.
\end{equation*}

A set $V\subset \R^n$ is \emph{finely open} if
$\R^n \setminus V$ is thin at each point $x\in V$.
\end{deff}

It is easy to see that the finely open sets form
a topology, which is called the \emph{fine topology}.
Every open set is finely open, but the converse is not true if $sp \le n$.
For example, if $\Csp(E)=0$ then $\Om \setm E$ is finely open.

On the other hand, if $sp>n$ then every finely open set is open in the usual
topology.
To see this, note that if $x \in \clE$ and $r>0$, 
then there is $y \in E \cap B(x, r)$ and thus 
by \eqref{eq-cap-sp>n},
\begin{equation*} 
\csp(E \cap B(x,r),B(x,2r))
\geq \csp(\{y\},B(y,3r)) 
\simeq r^{n-sp}.
\end{equation*} 
Inserting this into the Wiener integral shows that
$E$ is not thin at $x$.

Thinness and the fine topology depend on $s$ and $p$ in the same
way as regularity does, as described in Theorem~\ref{thm-reg-inclusion-intro}
and Section~\ref{sect-different-sp}.
This was actually what Adams--Hedberg~\cite{AH84} showed, based on the 
Wiener integral~\eqref{eq-int-Csp=infty} with the Sobolev capacity.

A function 
$u : \Om \to \eR$
is 
\emph{finely continuous} if it is continuous when $\Om$ 
is equipped with the
fine topology and $\eR$ with the usual topology.
(Recall that $\Om$ is always open (with respect to the Euclidean topology)
in this paper.)
Again, if $sp>n$ then every finely continuous function is 
$\eR$-valued continuous 
with respect to the usual topology.   

A closely related concept is quasicontinuity.

\begin{deff} \label{deff-qcont}
A function $u : \Om\to \eR$ is 
\emph{quasicontinuous} in $\Om$
if for every $\eps>0$ there is an open set $G$
with $\Csp(G)<\eps$ such that $u|_{\Om \setm G}$ is continuous (and finite-valued).
\end{deff}

Note that every quasicontinuous function is continuous when $sp>n$,
by Lemma \ref{lem-sp<=n}.
Since our Sobolev capacity $\Csp$ is comparable to the Bessel potential
capacity in 
Adams--Hedberg~\cite{AH} (see Remark~\ref{rmk-Bessel-cap}),
the notion of quasicontinuity is also the same.
Moreover, it follows from Theorem~\ref{thm-equiv-Wiener-int}
that the thinness and fine topology
defined above are 
the same as the ones considered
in Adams--Hedberg~\cite[(1)]{AH84}, \cite[Definition~6.3.7]{AH}.
In particular, the so-called Kellogg and Choquet  properties
for the fine topology follow directly from 
\cite[Corollary~6.3.17 and Theorem~6.3.18]{AH}.
Also the following characterization 
is a consequence of results in \cite{AH}.
A property holds \emph{quasieverywhere} (q.e.)\
if the set of points  for which it fails
has Sobolev capacity zero.

\begin{thm} \label{thm-qcont<=>finecont}
\textup{(\cite[Theorems~6.4.5 and~6.4.6]{AH})}
A function $u:\Om \to \eR$ is quasicontinuous
if and only if it is finely continuous q.e.\ and finite q.e.
\end{thm}

It is well known that Sobolev functions have ``better'' than 
a.e.-representatives.
We can
now apply results from Adams--Hedberg~\cite{AH} (and earlier papers) to obtain
the following result.

\begin{prop} \label{prop-Wsp-better-repr}
If $u \in \Wsploc(\Om)$, 
then there is a quasicontinuous ``representative'' $v$ such that
$v=u$ a.e.\ in $\Om$.
Moreover, $v$ is finely continuous q.e.\ in $\Om$ and finite q.e.\ in~$\Om$.
\end{prop}

This is well known for 
$u \in \Wsp(\Rn)$,
see Adams--Hedberg~\cite[Propositions~4.4.1, 6.3.10,  Theorems~6.4.5, 6.4.6]{AH} 
(and Remark~\ref{rmk-Bessel-cap})
for $sp \le n$ and 
Triebel~\cite[(5), p.~203]{Triebel95}
for  $sp>n$.
For general $\Om$ it then follows by localization.
Since $\Wsp$ is a nonlocal space
and the following
Proposition~\ref{prop-Kilp-qcont} is needed in the proof, 
we will provide the details for the reader's convenience
at the end of this section.

\begin{prop} \label{prop-Kilp-qcont}
\textup{(Adams--Hedberg~\cite[Theorem~6.1.4]{AH})}
If $u_1$ and $u_2$ are quasicontinuous in $\Om$ and $u_1=u_2$ a.e.\ in $\Om$,
then $u_1=u_2$ q.e.\ in $\Om$.
\end{prop}

It follows from~\eqref{eq-cap-E},
and the subadditivity of the capacity, that  changing 
a quasicontinuous function in a set of zero capacity 
preserves quasicontinuity.   
Proposition~\ref{prop-Kilp-qcont}
is a kind of converse of this fact.
We seize the opportunity to present   
a more elementary proof,
due to
Kilpel\"ainen~\cite{kilp98}.
Here $|\cdot|$ denotes the Lebesgue measure.

\begin{lem} \label{lem-Kilp-cond}
If $G \subset \Rn$ is open and $|E|=0$, then
\[
      \Csp(G)=\Csp(G \setm E).
\]
\end{lem}

\begin{proof}
Let $\eps>0$.
By Proposition~\ref{prop-Wsp-Csp-1} and truncation,
there is $u \in \Wsp(\Rn)$ such that
$0 \le u \le 1$ everywhere, $u=1$ 
in  an open set containing $G \setm E$
and $\|u\|^p_{\Wsp(\Rn)} < \Cpt(G \setm E) + \eps$.
Then $v:=\max\{u,\chi_E\}=u$ a.e., and thus $v \in \Wsp(\Rn)$.
Since $v=1$ in $G$ and $G$ is open,
it again follows from  Proposition~\ref{prop-Wsp-Csp-1} that 
\[
 \Cpt(G) \le \|v\|^p_{\Wsp(\Rn)} = \|u\|^p_{\Wsp(\Rn)}  <  \Cpt(G \setm E) + \eps.
\]
Letting $\eps \to 0$ shows that  
$\Cpt(G) \le   \Cpt(G \setm E)$.
The reverse
inequality follows from monotonicity.
\end{proof}

\begin{proof}[Proof of Proposition~\ref{prop-Kilp-qcont}]
Let $\eps >0$.
As $u$ and $v$ are quasicontinuous  in $\Om$,
we can find an open set $G \subset \Om$ with $\Csp(G)< \eps$
such that
$u|_{\Om \setm G}$ and $v|_{\Om \setm G}$ are continuous.
Thus the set $\{x \in \Om \setm G : u(x) \ne v(x)\}$ is open 
in the relative topology on $\Om \setm G$,
i.e.\ there is an open set $U \subset \Om$
such that 
\[
   U \setm G = \{x \in \Om \setm G : u(x) \ne v(x)\}.
\]
Since $U \cup G$ is open and $|U \setm G|=0$, 
Lemma~\ref{lem-Kilp-cond} with $G$ and $E$ replaced by $U\cup G$ and
$U\setm G$ shows that 
\[
    \Csp(\{x \in \Om : u(x) \ne v(x)\})
    \le \Csp(U \cup G) = \Csp(G) < \eps.
\]
Letting $\eps \to 0$ completes the proof.
\end{proof}

\begin{lem}  \label{lem-cutoff}
\textup{(Kim--Lee~\cite[Lemma~2.8]{KL-rem})}
If $u \in \Wsploc(\Om)$ and 
$\eta \in C_c^\infty(\Om)$,
then $v:= u\eta \in \Wsp(\Rn)$,
where  $v$ is extended by $0$ on $\Omc$.
\end{lem}

\begin{proof}
In \cite[Lemma~2.8]{KL-rem} it was shown that $v \in \Wsp(\Om)$.
Since $\supp v \Subset \Om$, it is easily seen
that $v \in \Wsp(\Rn)$.
\end{proof}

\begin{proof}[Proof of Proposition~\ref{prop-Wsp-better-repr}]
Let $\eps>0$ and
cover $\Om$ by countably many balls $B^j  \Subset \Om$ 
such that  $\Om =\bigcup_{j=1}^\infty B^j$.
For each $j=1,2,\dots$\,, 
it follows from Lemma~\ref{lem-cutoff}
that there is a function $u_j \in \Wsp(\Rn)$
such that $u_j=u$ in $B^j$.

By Adams--Hedberg~\cite[Propositions~6.3.10 and 4.4.1]{AH} and
Remark~\ref{rmk-Bessel-cap},
there is  a quasicontinuous function $v_j=u_j$ a.e.\ in $\Rn$.
(Note that ${\cal G}_\al$ in the first printing of~\cite[pp.~167--168]{AH} 
should be ${\cal H}_\al$,
see the corrected second 
printing of \cite{AH} from 1999.)
In particular, there is  an open  set $G_j$ with $\Csp(G_j)<2^{-j-1} \eps$
such that $v_j|_{B^j \setm G_j}$ is continuous.
Let 
\[ 
v(x)=v_{j_x}(x), \ x \in \Om, 
\quad\text{where }j_x=\min\{j : x \in B^j\}.
\] 
Note that  $v_j=v_i$ a.e.\ in $B^j \cap B^i$, and since 
both functions are
quasicontinuous, it follows from 
Proposition~\ref{prop-Kilp-qcont} that $v_j=v_i$ q.e.\ 
in $B^j \cap B^i$.
Hence the set
\[
E_j:=\{x \in B^j: v_j(x)\ne v(x)\}
  \subset \bigcup_{i=1}^{j-1} \{x \in B^j: v_j(x)\ne v_{i}(x)\}
\]
has capacity zero, by the subadditivity of the capacity.
By \eqref{eq-cap-E} there is an open set 
$G_j' \supset E_j$ such that 
$\Csp(G_j')<2^{-j-1} \eps$.
Let $G=\bigcup_{j=1}^\infty (G_j \cup G_j')$. 
Then $G$ is open 
 and $\Csp(G)<\eps$, by the countable subadditivity of the capacity.
Moreover, $v|_{\Om \setm G}$ is continuous. 
Thus $v$ is quasicontinuous in $\Om$.
Since $u=v$ a.e.\ in $\Om$, this concludes the proof
of the first part.

The last part now follows directly from
Theorem~\ref{thm-qcont<=>finecont}.
\end{proof}

\section{\texorpdfstring{$\LL$}{L}-superharmonic functions and polar sets}
\label{sect-polar}

\emph{Recall the general assumptions from the beginning
of Section~\ref{sec-sobolev}.}

\medskip

\begin{deff} \label{def:superharmonic}
A measurable function $u: \R^n \to \eR$ is \emph{$\LL$-superharmonic}  
in $\Omega$ if it satisfies the following properties:
\begin{enumerate}
\item \label{s-a}
$u < \infty$ almost everywhere in $\R^{n}$ and $u>-\infty$ everywhere in $\Omega$,
\item
$u$ is lower semicontinuous  in $\Omega$,
\item
for each open set $G \Subset \Omega$ and each solution $v \in C(\clG)$ of $\mathcal{L}v=0$ 
in $G$ satisfying $v_+ \in L^{\infty}(\R^{n})$ and $v \le u$ on  $\Gc$,
it holds that $v \le u$ in $G$,
\item \label{s-d}
$u_- \in L^{p-1}_{sp}(\R^{n})$.
\end{enumerate}
\end{deff}

It follows directly from the definition that $\min\{u,l\}$
is $\LL$-superharmonic in $\Om$ whenever $u$ is $\LL$-superharmonic in $\Om$
and $l \in \R$.

The following result from Korvenp\"a\"a--Kuusi--Palatucci~\cite{KKP17}
explains the close connection between supersolutions and $\LL$-superharmonic functions.
Here 
the \emph{lsc-regularization} of $u$ in $\Om$ is
\begin{equation*}  
\uhat(x) :=   \begin{cases}
   u(x), & \text{if } x \in \Omc, \\
\displaystyle\essliminf_{y\to x} u(y), & \text{if } x \in \Om,
\end{cases}
\end{equation*}
and $u$ is \emph{lsc-regularized} (in $\Om$) if $u \equiv \uhat$.

\begin{thm} 
\label{thm:KKP17}
\textup{(\cite[Theorems~1, 9 and~12]{KKP17})}
If $u$ is $\LL$-superharmonic  in $\Omega$, then $u$ is lsc-regularized in $\Om$.
If, in addition, $u$ is locally bounded in $\Omega$ or $u \in W^{s, p}_{\mathrm{loc}}(\Omega)$, 
then it is a  supersolution in $\Omega$.

Conversely, if $u$ is a supersolution in $\Om$, then $\uhat=u$ a.e.\
and $\uhat$ is $\LL$-superharmonic in $\Om$.
\end{thm}

The following lemma and its corollary
will be useful for studying local properties of $\LL$-superharmonic functions
and for proving Theorem~\ref{thm-superh-finecont}.

\begin{lem}   \label{lem-comp-u-ub}
Assume that $0\le u\le l$ in $B:=B(x_0,r)$ 
and $u_- \in L^{p-1}_{sp}(\Rn)$.
Then there is $M>0$ such that the function
\[
\ub := \begin{cases}   u & \text{in } 2B, \\
         -M & \text{in } \R^n \setm 2B, \end{cases}
\]
satisfies
\begin{equation}   \label{eq-comp-u-ub}
I(u,x) \le I(\ub,x)\quad \text{for }x \in B,
\end{equation}
where
\begin{equation}   \label{eq-comp-u-ub-I}
I(u,x) := \int_{(2B)^c} |u(x)-u(y)|^{p-2} (u(x)-u(y)) k(x, y) \,dy.
\end{equation}
\end{lem}

\begin{proof}
Since $0\le u \le l$ in $B$ and $|x-y| \ge r$, the left-hand side in~\eqref{eq-comp-u-ub} 
can be for $x\in B$ and $y\in (2B)^c$ estimated as follows
\begin{align*}
I(u,x) &\le \La  \int_{(2B)^c} (l + u_-(y))^{p-1}  \frac{dy}{|x-y|^{n+sp}} \\
 &\simle \int_{(2B)^c} (l + u_-(y))^{p-1}  \frac{dy}{(1+|y|)^{n+sp}} =:A, 
\end{align*}
where the comparison constant depends on $r$ but not on $x$.
Note that $u_- \in L^{p-1}_{sp}(\R^n)$ implies that $A<\infty$.
On the other hand, for the right-hand side in~\eqref{eq-comp-u-ub} we have
\begin{align*}
I(\ub,x) &= \int_{(2B)^c} (u(x)+M)^{p-1} k(x, y) \,dy \\
   &\ge M^{p-1} \La^{-1}
\int_{B(x,3r)^c} \frac{dy}{|x-y|^{n+sp}} \ge I(u,x),
\end{align*}
when $M$ is sufficiently large.
Note that $M$ depends on $r$ but not on $x$.
\end{proof}

\begin{cor}   \label{cor-ub-supersol}
Let $u$ be a supersolution in a ball $B$ such that $0\le u\le l$ in $B$.
Let $\ub$ and $M$ be as in Lemma~\ref{lem-comp-u-ub}.
Then $\ub$ is 
a  supersolution
in $B$. 
\end{cor}

\begin{proof}
Clearly, $\ub \in W^{s, p}_{\mathrm{loc}}(B) \cap L^{p-1}_{sp}(\R^n)$.
We shall show that $\EE(\ub,\phi) \ge0$ for all 
$0 \le \phi \in C_c^{\infty}(B)$. 
Since $\ub=u$ in $2B$ and $\phi=0$ outside $B$, we have
\begin{align*}
\EE(\ub,\phi) &= \int_{2B} \int_{2B} |u(x)-u(y)|^{p-2} (u(x)-u(y))(\phi(x)-\phi(y)) k(x,y) \,dy\,dx \\
& \quad + \int_{B} \int_{(2B)^c} |\ub(x)-\ub(y)|^{p-2} (\ub(x)-\ub(y)) \phi(x) k(x,y) \,dy\,dx \\
& \quad + \int_{(2B)^c} \int_{B} |\ub(x)-\ub(y)|^{p-2} (\ub(y)-\ub(x)) \phi(y) k(x,y) \,dy\,dx,
\end{align*}
where the last two integrals are equal by the symmetry of $k$ and can be written as
\[
\int_{B} I(\ub,x) \phi(x) \,dx,
\]
with $I(\ub,x)$ as in~\eqref{eq-comp-u-ub-I}. 
Since $u$ is a supersolution in $B$, we have that 
\begin{align*}
0\le \EE(u,\phi) &= \int_{2B} \int_{2B} |u(x)-u(y)|^{p-2} (u(x)-u(y))(\phi(x)-\phi(y)) k(x, y) \,dy\,dx \\
& \quad + 2 \int_{B} \int_{(2B)^c} |u(x)-u(y)|^{p-2} (u(x)-u(y)) \phi(x) k(x,y) \,dy\,dx.
\end{align*}
Lemma~\ref{lem-comp-u-ub} shows that the last integral is majorized by the corresponding integral with $\ub$
and hence
\[
0 \le \EE(u,\phi) \le \EE(\ub,\phi).
\]
Thus $\ub$ is a supersolution in $B$.
\end{proof}

We are now ready to prove Theorem~\ref{thm-superh-finecont} for $sp \le n$.
For $sp>n$, we will need Proposition~\ref{prop-polar-iff}\ref{polar-a} when 
proving Theorem~\ref{thm-superh-finecont}. 
We therefore postpone that case to the end of this section.

\begin{proof}[Proof of Theorem~\ref{thm-superh-finecont} for $sp \le n$]
By Lemma~\ref{lem-sp<=n}, $\Csp(\{x\})=0$ for $x\in\R^n$.
Let $l$ be an arbitrary real number.
By the lower semicontinuity of $u$, the set $\{x\in\Om:u(x)\le l\}$
is closed (and hence finely closed).
We shall show that the set 
\[
E := \{x\in\Om:u(x)\ge l\}
\]
 is also finely closed.
By the lower semicontinuity of $u$, $E$ is a $G_\de$ set and in particular a Borel set.

It suffices to show that $E$ is thin at every $x\in\Om$ for which
$u(x)<l$.
Consider such an $x$ and find a ball $B=B(x,r)$
such that $2B\Subset\Om$. 
As $u$ is lower semicontinuous and $u>-\infty$ in $\Om$, it is locally bounded
from below in $\Om$.
By adding a constant, we may therefore assume that $u\ge0$ in $2B$
and   $l >0$.
Then $u_l=\min\{u,l\}$ is an $\LL$-superharmonic function in $\Om$, 
and hence a supersolution in $\Omega$ by Theorem~\ref{thm:KKP17}.
Corollary~\ref{cor-ub-supersol} implies
that for some $M\ge0$, the function
\[
\ub := \begin{cases}   u_l & \text{in } 2B, \\
         -M 
& \text{in } \R^n \setm 2B, 
   \end{cases}
\]
is a supersolution
in $B$.  
As $u$ is lsc-regularized in $B$, so is $\ub$.
Thus $\ub$ is also $\LL$-superharmonic in $B$, by Theorem~\ref{thm:KKP17}.
Replacing $\ub$ and $l$ by $\ub+M$ and $l+M$, respectively, 
we may assume that $\ub \ge0$ in $\R^n$, $l>0$ and $M=0$.

Assume that $E$ is not thin at $x$. 
Then the Wiener integral in Definition~\ref{deff-thinness} diverges.
By Remark~\ref{rmk-Wiener-sum}, this is equivalent to
\[
\sum_{j=1}^\infty \biggl(\frac{\csp(E\cap B^j,2B^j)}{r_j^{n-sp}}  \biggr)^{1/(p-1)}  = \infty,
\]
where $B^j=B(x,r_j)$ and 
$r_j=2^{-j}r$, $j=1,2,\dots$\,.
Since $E$ is a Borel set and $\csp$ is a Choquet
capacity, it follows from 
Theorem~\ref{thm-cap-inner-reg}
that we can find compact sets 
$K_j\subset E\cap B^j$
such that
\[
\csp(K_j,2B^j) \ge \tfrac12 \csp(E\cap B^j,2B^j),
\quad j=1,2,\dots.
\]
Then 
$K := \{x\} \cup \bigcup_{j=1}^\infty K_j \subset B$
is a compact set  and
\begin{equation}   \label{eq-choose-K}
\sum_{j=1}^\infty 
    \biggl(   \frac{\csp(K \cap B^j,2B^j)}{r_j^{n-sp}} \biggr)^{1/(p-1)} 
  \simge 
\sum_{j=1}^\infty 
 \biggl(   \frac{\csp(E \cap B^j,2B^j)}{r_j^{n-sp}} \biggr)^{1/(p-1)} 
= \infty.
\end{equation}
Next,  let $G= B\setm K$.
Consider a Dirichlet data $g\in \Lip_c(B)$ such that $g=l$ on~$K$.
Then by the lower semicontinuity of $\ub$,
\[
\liminf_{G\ni y\to z} \ub(y) \ge \ub(z) = l = g(z) \quad \text{for every }
z\in \bdy K \setm \{x\}
\]
and, since $\ub \ge 0$,
\[
\liminf_{G\ni y\to z} \ub(y) \ge 0 = g(z) \quad \text{for every } 
z\in \bdy B.
\]
Moreover,
\[
\ub \ge \begin{cases}   
        l=g & \text{in } K \setm \{x\}, \\
        0=g & \text{in } \R^n\setm B,
   \end{cases}
\]
that is, $\ub\ge g+h$ in $\R^n\setm G$, where $h=-l\chi_{\{x\}}$.
Thus, $\ub$ is admissible in the definition of the upper Perron
solution $\uP_G(g+h)$ in $G$, see Definition~1.1 in~\cite{BBK1}.
Since $\Csp(\{x\})=0$ (by Lemma~\ref{lem-sp<=n}), it follows from
 Theorem~9.4 in~\cite{BBK1} 
(with  $\Om$ replaced by $G$)
that
\[
\ub \ge \uP_G(g+h) = H_G g    \quad \text{in } G,
\]
where 
$H_Gg$ is the Sobolev solution as in Theorem~\ref{thm-ex-Hg}.
Since $K$ is not thin at~$x$, by \eqref{eq-choose-K},
the Wiener criterion (Theorem~\ref{thm-Wiener}),
together with the facts that $\ub$ is lsc-regularized
and 
$g=l$ in $K$ 
implies that 
\[
u_l(x) = \ub(x) 
= \essliminf_{B\ni y\to x} \ub(y)
= \essliminf_{G\ni y\to x} \ub(y) 
\ge \liminf_{G\ni y\to x} H_Gg(y) = g(x) = l,
\]
which contradicts the assumption $u(x)<l$ and concludes the proof.
\end{proof}

To prove Proposition~\ref{prop-polar-iff}\ref{polar-a},
we will use the following important logarithmic estimate
for supersolutions.
In \cite{DCKP16}, it is stated for 
$u \in \Wsp(\Rn)$ but the proof therein only uses that 
$u \in \Wsploc(2B) \cap L^{p-1}_{sp}(\Rn)$.
(Here we apply the estimate from~\cite{DCKP16}
to $u-1$ and with $d=1$.)

\begin{lem} \label{lem-DCKP-log}
\textup{(Di Castro--Kuusi--Palatucci~\cite[Lemma~1.3]{DCKP16})}
Let $B=B(x_0,r)$ be a ball
and  $u$
be a supersolution in $2B$ such that $u \ge 1$ in $2B$.
Let $v=\log u$ in  $B$. 
Then
\[
    [v]_{\Wsp(B)}^p 
    \simle  r^{n-sp}[1+\Tail((u-1)_-; x_0,2r)^{p-1}],
\]
where
\begin{equation*}
\Tail(f; x_0, r) = \biggl( r^{sp} \int_{\Rn \setminus B(x_0,r)} 
    \frac{|f(y)|^{p-1}}{|y-x_0|^{n+sp}} \,dy \biggr)^{1/(p-1)}.
\end{equation*}
\end{lem}

\begin{lem} \label{lem-log-superh}
Let $B=B(x_0,r)$ be a ball
and  $u$ be an $\LL$-superharmonic  function in $2B$ 
such that $u \ge 1$ in $2B$.
Then $v:=\log u\in \Wsp(B)$.
\end{lem}

\begin{proof}
Let $u_j=\min\{u,j\}$ in $\Rn$ and $v_j=\log u_j$ in $B$,
for positive integers $j$.
Then $u_j$ is a supersolution, by Theorem~\ref{thm:KKP17}.
Note that
\[
|u_j(x)-u_j(y)| \le |u_{j+1}(x)-u_{j+1}(y)|,
\]
by truncation.
It thus follows from
monotone  convergence
and Lemma~\ref{lem-DCKP-log} that
\[
   [v]_{\Wsp(B)}^p = \lim_{j \to \infty} [v_j]_{\Wsp(B)}^p
    \simle  r^{n-sp}[1+\Tail((u-1)_-; x_0,2r)^{p-1}] < \infty,
\]
since $(u-1)_- \in L_{sp}^{p-1}(\Rn)$ by 
Definition~\ref{def:superharmonic}\ref{s-d}.
Moreover, 
$u< \infty$ a.e.\ in $\Rn$ (by Definition~\ref{def:superharmonic}\ref{s-a}),
and 
there is thus $M \ge 1$ such that  the measure of
$E:=\{x \in B : |v(x)|<M\}$ is positive.
Since $|v(x)| \le  |v(x)-v(y)| + M$ 
for every $y\in E$, it then follows that 
\begin{align*}
\int_B |v|^p \, dx 
    & \simle   M^p |B| + \frac{r^{n+sp}}{|E|}   \int_{B}  \int_E 
                   \frac{|v(x)-v(y)|^p}{|x-y|^{n+s p}} \, dy\, dx \\
   &\le M^p|B| + \frac{r^{n+sp}}{|E|}[v]_{\Wsp(B)}^p
< \infty.
\qedhere
\end{align*}
\end{proof}

\begin{proof}[Proof of Proposition~\ref{prop-polar-iff}]
\ref{polar-a}
Fix a ball $B=B(x,r)$ such that $4B \Subset \Om$.  
Since $u$ is lower semicontinuous in $\Om$ and $>-\infty$ in $\Om$,
it is bounded from below in $4B$. 
We may assume that $u \ge 1$ in $4B$. 
Then $v:=\log u \in \Wsp(2B)$, by Lemma~\ref{lem-log-superh}.
It follows from  Lemma~\ref{lem-cutoff}
that there is a function $\vt \in \Wsp(\Rn)$
such that $\vt=v$ in~$B$.

Let $E=\{x \in \Om:u(x)=\infty\}$.
Then $\vt=\infty$ in $B \cap E$.
The set $G_l:=\{x \in B : \vt(x)>l\}$  is open,
because $u$ (and thus also $\vt$) is lower semicontinuous in $B$.
It follows that $\vt/l$, with $l>0$, is admissible for  $\Csp(G_l)$, 
as in Proposition~\ref{prop-Wsp-Csp-1}, and thus
\[
     \Csp(B \cap E) 
\le \Csp(G_l)
   \le \|\vt/l\|^p_{\Wsp(\Rn)}
  = \frac{1}{l^p}  \|\vt\|^p_{\Wsp(\Rn)}
   \to 0, \quad \text{as } l \to \infty.
\]
Hence $\Csp(B \cap E)=0$.
Using a countable cover by balls $4B_j \Subset \Om$ with
$\Om=\bigcup_{j=1}^\infty B_j$, together with the countable subadditivity of the 
capacity, concludes the proof of~\ref{polar-a}.

\ref{polar-conv}
By the definition \eqref{eq-cap-E} of $\Csp(E)$,  there are open sets
$G_j$  such that $E \subset G_j \subset \Om$ and $\Csp(G_j)< 2^{-jp}$, $j=1,2,\dots$\,.
For each $G_j$, it follows from Proposition~\ref{prop-Wsp-Csp-1} that
there is 
a function $\psi_j$ 
such that
$\psi_j\ge 1$ on $G_j$ and $\|\psi_j\|_{\Wsp(\R^n)}< 2^{-j}$.
Let 
\[
\phi= \sum_{j=1}^{\infty} (\psi_j)_+ \in \Wsp(\Rn).
\]
By Kim--Lee~\cite[Theorem~4.9]{KL}, there exists an lsc-regularized solution $u$
to the $\K_{\phi,\phi}(\Om)$-obstacle problem.
What is important here is that $u$ is an lsc-regularized  supersolution in $\Om$
and that $u \ge \phi$ a.e.
By Theorem~\ref{thm:KKP17}, $u$ is $\LL$-superharmonic in~$\Om$.

Moreover,  $u \ge \phi \ge l$ a.e.\ in 
the open set $\bigcap_{j=1}^{l} G_j \supset E$ 
and thus
(since $u$ is lsc-regularized)
$u \ge l$ everywhere therein, and in particular in $E$.
Letting $l \to \infty$ therefore shows that $u=\infty$ in $E$.
\end{proof}

\begin{proof}[Proof of Theorem~\ref{thm-superh-finecont} for $sp > n$]
Let $x_0 \in \Om$.
By the definition of $\LL$-super\-har\-mon\-i\-city,  $u(x_0) > -\infty$ in $\Om$.
Lemma~\ref{lem-sp<=n} shows that
$\Csp(\{x_0\})>0$ and thus
$u(x_0) < \infty$, by Proposition~\ref{prop-polar-iff}\ref{polar-a}.   
Let $v(x)=\min\{u(x),u(x_0)+1\}$,
which is an lsc-regularized supersolution in $\Om$,
by Theorem~\ref{thm:KKP17}.
In particular $v \in \Wsploc(\Om)$.
As $sp>n$, 
it is well known
that there is a continuous representative $\vt \in \Wsploc(\Om) \cap C(\Om)$ such that
$\vt =v$ a.e.,
see e.g.\  Triebel~\cite[(5), p.~203]{Triebel95}. 
Since $v$ is lsc-regularized, we see that
\[
    v(x) 
= \essliminf_{y\to x} v(y) 
= \essliminf_{y\to x} \vt(y) 
= \vt(x) 
\quad \text{for } x \in \Om.
\]
Hence $v$  is continuous in $\Om$, and thus $u$ is continuous at $x_0$.
\end{proof}

If $u$ is $\LL$-superharmonic in $\Om$,
then $E:=\{x \in \Om:u(x)=\infty\}$ is a $G_\de$-set, since
$u$ is lower semicontinuous in $\Om$.
The following is therefore a natural question.

\begin{openprob}\label{open-prob}
Let $E \subset \Rn$ be a $G_\de$-set with $\Csp(E)=0$.
Is there then an open set $\Om \supset E$ 
and an $\LL$-superharmonic function in $\Om$ such that
$E=\{x \in \Om:u(x)=\infty\}$?
\end{openprob}

In the local nonlinear case
for $\A$-harmonic functions of \p-Laplacian type, 
Kil\-pe\-l\"ai\-nen~\cite{Kilp99} 
showed this with $\Om=\Rn$.

\section{Theorem~\ref{thm-reg-inclusion-intro} and regularity
for different   \texorpdfstring{$(s,p)$}{(s,p)}}
\label{sect-different-sp}

\emph{Recall the general assumptions from the beginning
of Section~\ref{sec-sobolev}.}

\medskip

In this section, we will
also include the local case 
$\Delta_pu=0$  
in our investigations,
which is relevant for the comparison results below.
For notational purposes, we make the convention that
this corresponds to $s=1$, that is, it regards the Sobolev space
$\Wp$, even though the kernel $k(x,y)$ with $s=1$ does not make sense here.
In this case we define 
$C_{1,p}$  and $\capp_{1,p}$
to be the Sobolev (resp.\ condenser) capacity 
associated with the classical  
Sobolev (semi)norm 
\[
\|u\|_{\Wp(\Rn)}^p := \|u\|_{L^p(\Rn)}^p + [u]_{\Wp(\Rn)}^p
\quad \text{resp.} \quad
 [u]_{\Wp(\Rn)}^p := \int_{\R^n} |\grad u|^p\,dx
\]
as in
Definitions~\ref{deff-cpt} and \ref{deff-Csp}, 
together with \eqref{eq-cap-G} and~\eqref{eq-cap-E}.
For $s=1$ we use the definition of regular boundary points
from 
e.g.\ Heinonen--Kilpel\"ainen--Martio~\cite[p.~171]{HeKiMa}
or Mal\'y--Ziemer~\cite[Definition~2.131]{MZ}.
The Wiener criterion of 
Maz{\cprime}ya~\cite[Theorem, p.~236]{Mazya} (sufficiency),
Lindqvist--Martio~\cite{LM85} (necessity for $p>n-1$) and
Kilpel\"ainen--Mal\'y~\cite[Theorem~1.1]{KiMa94} (necessity for all $p>1$)
shows that regularity is characterized by the corresponding
Wiener integral with a condenser capacity, see alternatively \cite[Theorem~21.30]{HeKiMa}.
It follows
from Theorem~2.49 in~\cite{MZ}
that it
can equivalently be characterized by a Wiener integral
with a Sobolev capacity (provided that $p \le n$).

Our aim is now to prove Theorem~\ref{thm-reg-inclusion-intro}.
Since the formulation is rather complicated we repeat it here.

\begin{thm}[$=$Theorem~\ref{thm-reg-inclusion-intro}]
\label{thm-reg-inclusion}

Consider $0 < s_j \leq 1 < p_j$, $j=1,2$, with $(s_1, p_1) \ne (s_2, p_2)$.
The implication
\begin{equation}   \label{eq-imp-reg-sp}
\text{$x_0$ is regular for $(s_1,p_1)$}
\imp
\text{$x_0$ is regular for $(s_2,p_2)$}
\end{equation}
holds, for all bounded open sets $\Om$ with $x_0 \in \bdy \Om$,
if and only if
any of the following  mutually disjoint cases holds\/\textup{:}
\begin{enumerate}
\item
  $s_2 p_2 > n$,
\item \label{f-b}
  $s_1 p_1 = s_2 p_2 = n$ and $p_1> p_2$,
\item \label{f-c}
  $s_1 p_1 < s_2 p_2 = n$,
\item \label{f-d}
  $s_1 p_1 < s_2 p_2 < n$ and
  \begin{equation}   \label{eq-cond-Wiener-AH}
  \frac{s_1(p_2-1)+n}{p_2}
  \le  
  \frac{s_2(p_1-1)+n}{p_1}.
  \end{equation}
\end{enumerate}  
\end{thm}

\begin{proof} 
Recall that regularity for $sp \le n$ is characterized
by the Wiener 
condition \eqref{eq-int-Csp=infty}
 by Theorems~\ref{thm-Wiener} and~\ref{thm-equiv-Wiener-int},
and also that the Bessel potential capacity
(denoted $C_{s,p}$ in~\cite{AH})
is comparable to our Sobolev 
capacity $\Csp$, see Remark~\ref{rmk-Bessel-cap}.
Thus the $(s,p)$-thinness defined 
in~\cite[Definition~6.3.7]{AH} 
(for $sp \le n$)
is equivalent to the thinness defined
using~Definition~\ref{deff-thinness}.
A similar equivalence holds for $s=1$, see 
the discussion before the statement of Theorem~\ref{thm-reg-inclusion}.

The implication holds trivially when $s_2p_2>n$ since
points have positive capacity, and thus are regular,
by Lemma~\ref{lem-sp<=n} and Theorem~\ref{thm-Wiener}.

In the cases~\ref{f-b}--\ref{f-d}, the implication~\eqref{eq-imp-reg-sp} 
follows directly from (the positive direction in)
Theorem~B in Adams--Hedberg~\cite{AH84} (taking $(\al,p)=(s_2,p_2)$
and $(\be,q)=(s_1,p_1)$ therein). 

For the converse, in each remaining case we need to find a
bounded open set $\Om$ such that $0 \in \bdy \Om$ and the implication fails.
If $s_1 p_1 > n \ge s_2 p_2$ then with $\Om=B(0,1) \setm \{0\}$,
the origin $0$ is regular for $(s_1,p_1)$ but not for $(s_2,p_2)$.

When $s_j p_j \le n$, $j=1,2$, the counterexamples constructed in~\cite[Section~4]{AH84}
are  not closed sets
and thus do not have open complements (corresponding to $\Om$ here).
However, by adding the origin $0$ to the sets from~\cite{AH84}
and in addition using closed balls in their case (b), one obtains  compact sets.
For the reader's
convenience, we briefly describe the construction from~\cite{AH84}
and modify it to our purposes. 
Since we can work directly with the condenser capacity $\csp$, 
the case (c) in~\cite[Section~4]{AH84} can be simplified and 
included in case (a) therein as follows.

When $s_2 p_2 = s_1 p_1\le n$ and $p_1 < p_2$, or $s_2 p_2 < s_1 p_1\le n$,
Theorem~5.5 in Adams--Meyers~\cite{AM73}
(or Theorem~5.5.1 in Adams--Hedberg~\cite{AH})
implies that there is a compact set 
$K\subset B(0,1)\setm \itoverline{B(0,\tfrac12)}$ 
such that
\[
 C_{s_2,p_2}(K)=0< C_{s_1,p_1}(K).
\]
By Lemma~\ref{lem-cp-Cp}, $\capp_{s_1,p_1}(K,B(0,2))>0$.
It follows easily from Definition~\ref{deff-cpt} 
(and the definition of $\capp_{1,p}$ when $s=1$)
by a change of variables that for all $\la>0$ and $0<s\le 1<p$,
\[
\csp(\la K,B(0,2\la)) = \la^{n-sp} \csp(K,B(0,2)), \quad
\text{where } \la K:=\{\la x: x \in K\}.
\]
For 
\begin{equation}    \label{eq-def-Om}
\Om = B(0,1)\setm \biggl( \{0\} \cup \bigcup_{j=0}^\infty 2^{-j} K \biggr)
\end{equation}
we then have $C_{s_2,p_2}(B(0,1)\setm\Om)=0$ and so the origin $0$
is irregular with respect to $(s_2,p_2)$, by the Wiener criterion.
At the same time, with $B^j:=B(0,2^{-j})$,
\[
\sum_{j=0}^\infty   \biggl(\frac{\capp_{s_1,p_1}(B^j\setm \Om,2B^j)}{2^{-j(n-s_1p_1)}}
          \biggr)^{1/(p-1)} 
          \ge \sum_{j=0}^\infty \capp_{s_1,p_1}(K,B(0,2))^{1/(p-1)} = \infty,
\]
 so $0$ is regular with respect to $(s_1,p_1)$, by the Wiener criterion.
 
 If $s_1 p_1 < s_2 p_2 < n$ and \eqref{eq-cond-Wiener-AH} fails,
 then a counterexample is obtained by 
replacing $2^{-j} K$ in \eqref{eq-def-Om} with the closed balls
\[
\itoverline{B(x_j,r_j)}\subset 2B^j\setm B^j, 
  \quad  \text{where } r_j= \frac{2^{-j}}{j^{(p_1-1)/(n-s_1p_1)}},
\]
see~\cite[Section~4]{AH84} for the details. 
\end{proof}

We conclude the paper by discussing the conditions in Theorem~\ref{thm-reg-inclusion}.
Fix $0<s_1\le 1< p_1<\infty$ and assume that a point $x_0$ is regular for $(s_1, p_1)$.
We shall 
visualize the set of exponents $s$ and $p$ from Theorem~\ref{thm-reg-inclusion}
for which $x_0$  then must be regular as well.
It will be convenient to formulate this in terms of $t_1=1/p_1$ and $t=1/p$.
Consider
the square $(0,1)\times(0,1]$.
Given a point $(t_1,s_1)$ in this square, we are interested in the set of exponents $(t,s)$
in the square
for which we get regularity from Theorem~\ref{thm-reg-inclusion}.
Those conditions are now reformulated as\/\textup{:}
\begin{enumerate}
\item  \label{sq-a}
  $s/t > n$,
\item \label{sq-b}
  $s_1/t_1 = s/t = n$ and $t_1    < t$,
\item \label{sq-c}
  $s_1/t_1 < s/t = n$,
\item \label{sq-d}
  $s_1/t_1 < s/t < n$ and
  \begin{equation}   \label{eq-cq-Wiener-AH}
 s_1(1-t) + nt  \le s(1-t_1) + nt_1.
  \end{equation}
\end{enumerate}  
From \ref{sq-a} we see that 
the exponents in the ``trivial'' domain  $s >nt$
strictly above the ``critical'' line $s=nt$
always give regularity.
Now assume that $(t_1,s_1)$ is on the critical 
line, i.e.\ $s_1=nt_1$.
Then by \ref{sq-b} also all the exponents on this line with $t>t_1$ 
(i.e.\ above the point $(t_1,s_1)$)
make $x_0$ regular.
It follows from Theorem~\ref{thm-reg-inclusion} that
these are all the exponents for which $x_0$ is always
regular
(in both cases $s_1 =nt_1$ and $s_1 < n t_1$).

Next we study the case when $(t_1,s_1)$ is below the critical line, i.e.\ $s_1<nt_1$.
By  \ref{sq-a} and  \ref{sq-c}, exponents in the relatively closed set $s\ge nt$ give regularity
(i.e.\ in the trivial domain and on the critical line).
In  \ref{sq-d} we must have $s/t>s_1/t_1$, so 
$s>(s_1/t_1) t$ and thus
$(t,s)$ must lie above the line
$s = (s_1/t_1) t$ through $(0,0)$ and $(t_1,s_1)$. 
Moreover, condition~\eqref{eq-cq-Wiener-AH} is equivalent to
\begin{equation}   \label{eq-def-f}
s \ge  \frac{s_1-nt_1}{1-t_1} + \frac{n-s_1}{1-t_1}t 
=: a_1 + b_1t =: f(t).
\end{equation}
This is a straight line which passes 
through the points $(t_1,s_1)$, $(0,a_1)$  and 
$(1,n)$.
The last
point also lies on the critical line $s=nt$
(but outside the square $(0,1) \times (0,1]$).
Note that $a_1<0$ since $s_1 < n t_1$.
Thus, regularity with respect to $(s_1,p_1)\in (0,1) \times (1,\infty)$
implies regularity for $(t,s)$ in the sector
\[
\{(t,s) \in (0,1)\times(0,1]:  s > (s_1/t_1)t \text{ and } s\ge f(t) \},
\]
marked ``reg'' in Figure~\ref{fig-one}.

A similar analysis shows that if $x_0$ is irregular for $(t_1,s_1)$ with $s_1<nt_1$, 
then it is irregular for all exponents $(t,s)$ in the sector 
\[
\{(t,s) \in (0,1)\times(0,1]:  s < (s_1/t_1)t \text{ and } s\le f(t) \},
\]
marked ``irreg'' in Figure~\ref{fig-one}.

Since Theorem~\ref{thm-reg-inclusion} is an ``if and only if'' statement, it follows that no implications 
between (ir)regularity with respect to $(s_1,p_1)$ and $(s,p)$ hold when $(1/p,s)$ belongs
to the sectors 
\[
\{(t,s): (s_1/t_1)t \le s < f(t)\} \quad \text{and}  \quad
\{(t,s): f(t) < s \le (s_1/t_1)t\},
\]
marked ``??'' in Figure~\ref{fig-one}.

\begin{figure}[htbp]   
  \centering
    \includegraphics{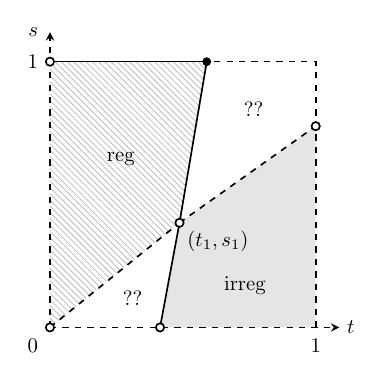}
  \caption{\label{fig-one}
This figure illustrates
when the implication \eqref{eq-imp-reg-sp} 
in Theorem~\ref{thm-reg-inclusion} holds
for $s_1 \le n t_1$, where $t_1=1/p_1$.
If $x_0 \in \bdy \Om$ is regular for $(t_1,s_1)$, then it is 
regular for all $(t,s)=(1/p,s)$ in the   
striped 
region marked ``reg''.
If on the other hand $x_0 \in \bdy \Om$ is irregular for $(t_1,s_1)=(1/p_1,s_1)$, 
then it is irregular for all $(t,s)$ in the grey region marked ``irreg''.
No implications hold for 
points $(t,s)$ in the white regions marked ``??''.
When $s_1 = n t_1$, the white regions marked ``??'' become empty.
}
\end{figure}

Finally, we analyse the relationship between the local case $s=1$ 
and the fractional cases $0<s<1$.
It follows from Theorem~\ref{thm-reg-inclusion} that regularity with respect to $(1,p_1)$ implies
regularity with respect to $(1,p)$ for all $p>p_1$.

From the above observations we conclude that regularity with respect to $(1,p_1)$ does not 
(apart from the trivial situation $sp>n$) imply regularity with respect to any $(s,p)\in (0,1) \times (1,\infty)$
when $p_1 \geq n$.
(In particular, regularity for $(1,n)$ is nontrivial, but not
strong enough to imply any nontrivial regularity for $(s,p)$ with $s<1$ 
and $sp \le n$.)
However, if $p_1<n$ (i.e.\ $t_1>1/n$), then 
regularity with respect to $(1, p_1)$
implies regularity with respect to any $(s, p)$ 
with $sp > p_1$.

Conversely, assuming regularity with respect to some 
$(s_1,p_1)\in (0,1) \times (1,\infty)$, 
regularity for $(1,p)$ is guaranteed
whenever one of the following conditions holds:
\begin{enumerate}
\renewcommand{\theenumi}{\textup{(\roman{enumi})}}%
\item
$p>n$ (by \ref{sq-a}), 
\item
$p=n\ge s_1p_1$ 
(by \ref{sq-b} or \ref{sq-c}),
\item
$1<p<n$ and $1\ge f(1/p)$ (by \ref{sq-d}) with $f$ as in \eqref{eq-def-f}.
\end{enumerate}

Note that if $s_2 \ge s_1$ and $p_2 \ge p_1$, then
\eqref{eq-imp-reg-sp} holds.


\end{document}